\documentclass[reqno,10pt]{amsart}
\usepackage{amsmath}
\usepackage{amssymb}
\usepackage{amsfonts}
\usepackage{graphicx}
\usepackage[abs]{overpic}
\usepackage{caption}
\usepackage{subcaption}
\usepackage{wrapfig}
\usepackage{float}
\usepackage{graphicx}
\usepackage{physics}
\usepackage{esint} 
\usepackage{braket}

\usepackage{tikz}
\usetikzlibrary{decorations.pathreplacing,calc}
\def\centerarc[#1](#2)(#3:#4:#5)
{ \draw[#1] ($(#2)+({#5*cos(#3)},{#5*sin(#3)})$) arc (#3:#4:#5); }

\definecolor{blue_links}{RGB}{13,0,180} 

\usepackage{hyperref}
\hypersetup{
    colorlinks=true, 
    linktoc=all,     
    linkcolor=blue_links,  
    citecolor=blue_links,
    urlcolor=blue_links,
}
\usepackage{cleveref}

\usepackage{mathtools}

\usepackage{xcolor}

\newtheorem{theorem}{Theorem}[section]
\newtheorem{lemma}[theorem]{Lemma}

\newtheorem{corollary}[theorem]{Corollary}
\newtheorem{definition}[theorem]{Definition}

\newtheorem{remark}[theorem]{Remark}

\newtheorem*{theorem*}{Theorem}

\def\de{\mathrm{d}}

\newcommand{\R}{\mathbb{R}}

\newcommand\ep{\varepsilon}

\def\eps{\varepsilon}
\def\epsilon{\varepsilon}

\def\dist{\operatorname{dist}}
\def\sdist{\operatorname{sdist}}
\def\Xint#1{\mathchoice
    {\XXint\displaystyle\textstyle{#1}}
    {\XXint\textstyle\scriptstyle{#1}}
    {\XXint\scriptstyle\scriptscriptstyle{#1}}
    {\XXint\scriptscriptstyle\scriptscriptstyle{#1}}
\!\int}
\def\XXint#1#2#3{{\setbox0=\hbox{$#1{#2#3}{\int}$}
\vcenter{\hbox{$#2#3$}}\kern-.5\wd0}}

\def\dashint{\Xint-}

\newcommand\iO{\int_\Omega}
\newcommand{\wkstr}{\xrightharpoonup{*}}

\renewcommand{\div}[1]{\nabla\cdot{#1}}

\numberwithin{equation}{section}

\begin{document} 

\title[Convergence of a heterogeneous Allen--Cahn equation to weighted MCF]{Convergence of a heterogeneous Allen--Cahn equation \\
to weighted mean curvature flow}
\author[L. Ganedi]{Likhit Ganedi} 
\address[Likhit Ganedi]{Institut f\"ur Mathematik, RWTH Aachen, Templergraben 55,
52062 Aachen, Germany}
\email{ganedi@eddy.rwth-aachen.de}
\author[A. Marveggio]{Alice Marveggio}
\address[Alice Marveggio]{Hausdorff Center for Mathematics, Universit\"at Bonn, Endenicher Allee 62, 53115 Bonn,
Germany}
\email{alice.marveggio@hcm.uni-bonn.de}
\author[K. Stinson] {Kerrek Stinson} 
\address[Kerrek Stinson]{Department of Mathematics, University of Utah, Salt Lake City, UT 84112, USA}
\email{kerrek.stinson@utah.edu}


\begin{abstract} 
We consider a variational model for heterogeneous phase separation, based on {a diffuse interface energy} {with moving wells}. 
{Our main result identifies} the asymptotic behavior of the first variation of {the phase field} energies as the width of the diffuse interface {vanishes}. {This convergence result allows us to deduce} a Gibbs--Thomson relation for heterogeneous surface tension{s}.
{Proceeding} from this information, we prove that (weak) solutions of the Allen–Cahn equation 
{with space dependent potential}
converge to a $BV$ solution of weighted mean curvature flow, under an energy convergence hypothesis. Additionally, relying on the relative energy technique, we establish a weak-strong uniqueness principle for {solutions of} weighted mean curvature flow.
\end{abstract}
\maketitle
\section{Introduction}
Many physical, chemical, and biological processes are driven by phase separation, wherein mixtures prefer to separate into nearly pure regions with surface-like boundaries. Extensive research has been conducted on the connection between diffuse interface models for phase separation, which are amenable to numerics, and their sharp interface analogues, where phase boundaries are represented by hypersurfaces. While this connection is relatively well understood in the homogeneous setting, realistic models may need to incorporate heterogeneities, which can naturally arise due to material properties or non-isothermal settings. Motivated by recent work for heterogeneous phase separation at the static level \cite{cristoferigravina,cristoferifonsecaganedi}, we study the sharp interface limit of a heterogeneous model for phase separation \emph{with moving wells} and rigorously derive its limiting evolution as the diffuse interface-width parameter vanishes. 

We begin by introducing the underlying free energy for the variational model. On a material domain $\Omega\subset \R^N$ with $C^2$-boundary, we consider the energy
\begin{equation}\label{eq:energyIntro}
E_\eps[u] : = \int_{\Omega}\frac{1}{\eps}W(x,u) + \frac{\eps}{2}|\nabla u|^2 \, \de x,
\end{equation}
where $\eps>0$ is a parameter proportional to interface width, $u\in H^1(\Omega)$ is a phase-field function, and $W : \Omega\times \R \to [0,\infty)$ is a double-well function satisfying $W(x,u) = 0$ if and only if $u = a(x)$ or $u = b(x)$ for $a,b \in C^2(\overline{\Omega})$ with $a<b$.
As in Bouchitt\'e's work \cite{Bouchitte}, one can see $W(x,u)$ as an alternative to considering a double-well function of the form $\bar W(T,u)$, where $T:\Omega\to \R$ is a spatially dependent temperature-field.
 The {($\eps$-rescaled)} $L^2$-gradient flow of this energy is the \emph{heterogeneous Allen--Cahn equation}
\begin{equation}\label{eqn:AChetero}
\left\{\begin{aligned}
	\partial_t u_\varepsilon & =  \Delta u_\varepsilon - \frac{1}{\eps^2} \partial_{u} W(x, u_\eps) \quad && \text{ in } \Omega \times [0,\infty), \\
	\nabla u_{\eps}\cdot n_\Omega  & = 0 \quad  && \text{ on } \partial \Omega \times [0,\infty).
	\end{aligned} \right.
\end{equation}
At the static level, the derivation of the limit of \eqref{eq:energyIntro} as $\eps\to 0$ is due to Bouchitt\'e \cite{Bouchitte}, who studied the $\Gamma$-limit of functionals of the form
$\int_{\Omega}\frac{1}{\eps}f(x,u,\eps \nabla u)\, \de x,$
and, in particular, showed that the $\Gamma$-limit of $E_\eps$ is the weighted perimeter functional
\begin{equation}\label{eqn:energyLimitIntro}
E[ u ] : = \begin{cases}\int_{\Omega} \sigma \, \de |\nabla \chi_A| & \text{ for } u  = b\chi_A + a(1-\chi_A) \text{ with }A\subset 
\Omega, \\
+\infty & \text{ otherwise},
\end{cases}
\end{equation}
where $$\sigma (x) : = \int_{a(x)}^{b(x)}\sqrt{2W(x,s)}\, \de s$$ and we refer the reader to Subsection \ref{subsec:notation} for our notation.
We then expect that solutions of the heterogeneous Allen--Cahn equation \eqref{eqn:AChetero} will converge to a weighted mean curvature flow governed by the spatially dependent surface tension $\sigma.$
More precisely, we say the sets $t\mapsto A(t)\subset \Omega$ evolve by \emph{weighted mean curvature flow} if they satisfy
\begin{equation}\label{eqn:wMCFstrong}
\left\{\begin{aligned}
	\sigma V &= \sigma H_A - \nabla \sigma \cdot n_{A} \quad &&\text{ along } \partial A\cap \Omega,\\
	\partial A \angle \partial \Omega & = 90^\circ \quad &&\text{ on } \partial \Omega \cap \overline{\partial A \cap \Omega},
	\end{aligned}\right.
\end{equation}
where  $V$ is the surface normal velocity (in the direction $n_A$), $H_A = -\div{n_{A}}$ denotes the scalar mean curvature, and $n_A$ is (a smooth extension of) the inner normal on $\partial A$. We remark that this evolution can be found as the formal gradient flow of the weighted perimeter functional \eqref{eqn:energyLimitIntro} with respect to the ${L^2(\Omega;\sigma\mathcal{H}^{N-1}\llcorner {\partial A})}$ metric. 

Introducing a {distributional (or $BV$)} solution concept for the weighted mean curvature flow \eqref{eqn:wMCFstrong}, we prove that weak solutions of the heterogeneous Allen--Cahn equation \eqref{eqn:AChetero} converge to a solution of \eqref{eqn:wMCFstrong}  under a suitable energy convergence hypothesis (see Theorem \ref{theo:convAC}). 
Following Luckhaus and Sturzenhecker's work \cite{LuckhausStur}, as also in \cite{HenselLaux-contact,kroemer_Laux_2021,LauxSimon,kimMelletWu2022}, we rely on the assumption that the time-integrated energies of the phase-field approximation \eqref{eq:energyIntro} converge to the weighted perimeter as $\eps\rightarrow 0$. 
 To obtain convergence of solutions, we first show that, in a certain sense, the first variation of the diffuse energy $E_\eps$ converges to the first variation of the weighted perimeter functional $E$ (see Theorem \ref{theorem:firstvar}). To derive this limit, we introduce a normalized well function 
$$W_{\rm n}(x,v) : = W(x,a(x) + \gamma(x) v), \quad \text{ with }\gamma = b - a, $$
which critically allows us to reduce the analysis to the case of fixed wells---though the energy landscape still depends on $x$---wherein we may take advantage of a differential inequality $|\partial_x \sqrt{W_{\rm n}(x,v)}|\leq C \sqrt{W_{\rm n} (x,v)}$ (see \eqref{eqn:derivative control} below). While we must explicitly assume this inequality, it is satisfied by prototypical double-well functions and may be thought of as a slight strengthening of a Lipschitz assumption for $\sqrt{W}$, which can be used to compare $\sigma$ at neighboring points.

As a direct consequence of the convergence of first variations, we may derive a Gibbs--Thomson relation in the heterogeneous setting (see \Cref{cor:GT}). Simply stated, considering minimizers $u_\eps$ of $E_\eps$ subject to a mass constraint, one finds that 
\begin{equation}\nonumber
\lambda_\eps  = \eps \Delta u_\eps - \frac{1}{\eps}\partial_u W(x,u) \quad \text{ in } \Omega,
\end{equation}
where $\lambda_\eps\in \R$ is a Lagrange multiplier accounting for the mass-constraint or, in more physical terms, is the thermodynamic chemical potential of the system. The Gibbs--Thomson relation says that, for the limiting region $A$ determined by $u_\eps \to b\chi_A + a(1-\chi_A)$ as $\eps\to 0$, we have
$$\lambda_0 : = \lim_{\eps \to 0} \lambda_\eps =  \frac{1}{\gamma} {\left(-\sigma H_{A} + \nabla \sigma \cdot n_A\right)} \quad \text{ on $\partial A\cap \Omega$}.$$
In particular, the weighted mean curvature of the limit interface is prescribed by the limit of chemical potentials.
We remark that in the homogeneous setting the Gibbs--Thomson relation was conjectured by Gurtin in \cite{GurtinConjecture} and then proven by Luckhaus and Modica in \cite{LuckhausModica}. Incorporating anisotropy and spatial heterogeneity in the gradient term, Cicalese et al. \cite{CicaleseNagasePisante} proved the relation for energies of the form $\int_{\Omega}\frac{1}{\eps}W(u) + \eps g(x,\nabla u) \de x$.

Furthermore, we investigate the uniqueness properties of distributional (or $BV$) solutions to weighted mean curvature flow. 
We show that this solution concept satisfies a weak-strong uniqueness property, meaning that: 
As long as a classical smooth solution to the weighted mean curvature flow exists, any $BV$-solution with the same initial datum must coincide with this flow (see Theorem \ref{theo:weakstrong}). 
Similar results for the (standard) mean curvature flow have recently been established by means of the relative energy technique \cite{FischerHenselLauxSimon,HenselLaux-bubble,HenselLaux-contact}. This technique is quite versatile, and we refer to \cite{Laux-volume} for the volume preserving case, to \cite{LauxStinsonUllrich22} for the anisotropic case, and to \cite{FischerHenselMarveggioMoser} for the analysis of flow beyond a circular singularity. In our setting, we adopt the same strategy and develop the relative energy method for weighted mean curvature flow. 

We make a couple brief remarks situating our result in the context of current research. Several results on the convergence of the (standard) homogeneous Allen-Cahn equation to mean curvature flow are available in the literature, ranging from formal asymptotic analysis to rigorous results for a variety of solution concepts. For instance, for the convergence to a smooth surface evolution, we refer to the works \cite{DeMottoniSchatzman,Chen} and, more recently, to \cite{FischerLauxSimon,FischerMarveggio}. 
To understand the long-term behavior past singularities, the following concepts of weak solutions have proven to be useful for understanding the singular limit: viscosity
solutions \cite{ChenGigaGoto,EvansSpruck,EvansSonerSouganidis}, Brakke’s varifold solutions \cite{Brakke,Ilmanen}, and, more recently, BV-varifold solutions \cite{KimTonegawa,StuvardTonegawa} and De Giorgi's varifold solutions \cite{HenselLaux-varifold}. Two recent results \cite{BungertLauxStinson,ChambolleDeGennaroMorini} have recovered weighted mean curvature flows, as in our paper, via minimizing movements type schemes.

 Much of the drive to understand heterogeneities in phase separation processes has revolved around homogenization. Interesting results by the first author and collaborators Cristoferi and Fonseca have derived a variety of $\Gamma$-limits for functionals of the form $\int_\Omega \frac{1}{\eps}W(x/\delta,u) + \eps |\nabla u|^2 \de x$ with $\delta\to 0$ \cite{cristoferi_homogenization_2019,cristoferifonsecaganedi,cristoferi2023homogenization}. Related work by Morfe and collaborators look at a variety of Allen--Cahn equations in the setting of homogenization, developing viscosity solution theory for the limit behavior of the related curvature flows \cite{Morfe2020_gamma,Morfe2020HomogenizationOT,FeldmanMorfe2023}. While we do not broach the difficult subject of homogenization in this work, we note that in contrast to the work of Morfe, we adopt an entirely energetic approach (as in \cite{HenselLaux-contact,HenselLaux-varifold}) that has the hope of accounting for multiple phases (cf. \cite{LauxSimon}). The $\Gamma$-limit of the multiphase (or vectorial) analogue of $E_\eps$ has been identified in \cite{cristoferigravina} and, in further, generality by \cite{cristoferifonsecaganedi}. While beyond the scope of this paper, identification of the multiphase dynamics will be a subject of future investigation.

\subsection*{Overview} The rest of the paper is organized as follows. We present the precise mathematical setting for our results in Section \ref{sec:assumptions}, which includes a variety of assumptions used throughout the paper. Convergence of the first variation of the energies is proven in Section \ref{sec:firstVar}, along with the subsequent Gibbs--Thomson relation.
In \Cref{sec:convAC}, we prove that {weak} solutions of the heterogeneous Allen--Cahn equation converge to a {distributional (or BV)} solution of weighted mean curvature flow. Finally, in \Cref{sec:weakStrong}, we prove that the limiting flow {satisfies} a weak-strong uniqueness principle.


\section{Mathematical setting}\label{sec:assumptions}
We briefly introduce notation used throughout the paper and then turn to the mathematical setting for our results.

\subsection{Notation} \label{subsec:notation}
{We denote by $N$ an integer number larger than $2$.} We use $\mathcal{L}^N$ to denote the Lebesgue measure. For a measure $\mu$, $f\mu$ is the measure absolutely continuous with respect to $\mu$ having density $f.$ We use $C>0$ for a generic constant which may change from line to line. We reserve $\Omega\subset \R^N$ as our domain of interest, and it will typically be a $C^2$-domain, meaning it is an open and bounded set with $C^2$-boundary. {We adopt the symbol $^\ast$ to denote dual spaces (e.g., $C(\overline \Omega)^\ast $).} The function $\chi_A$ is the characteristic associated to the set $A$. We say $A$ is a set of finite perimeter in $\Omega$ if $\chi_A\in BV(\Omega;\{0,1\}),$ and we denote its distributional derivative in $\Omega$ by $\nabla \chi_A$ and its total variation measure by $|\nabla \chi_A|$. We use $n_A:=\frac{\de \nabla \chi_A}{\de |\nabla \chi_A|}$ for the measure-theoretic inner normal of $A$. Using $\partial^\ast A$ as the reduced boundary of the set $A$ in $\Omega$, we have $\int_\Omega \cdot \, \de |\nabla \chi_A| =  \int_{\partial^\ast A} \cdot \, \de \mathcal{H}^{N-1}$, where $\mathcal{H}^{N-1}$ is the Hausdorff measure. For more information on functions of bounded variation, we refer to \cite{AmbrosioFuscoPallara}. 
Finally, we reserve the symbol $\partial_x$ to mean differentiation with respect to the $x$-input of a function, i.e., for $W:\Omega\times \R \to [0,\infty)$ and a function $u:\Omega\to \R$, we have $\partial_xW(x,u(x)) = (\partial_x W(x,\zeta))\circ (x,u(x))$.

\subsection{Setting}\label{subsec:setting}{For $N \geq 2$, $\Omega\subset \R^N$ a $C^2$-domain, and $\eps>0$,} we consider a heterogeneous Cahn--Hilliard (or Modica-Mortola) energy of the form
\begin{equation}\label{eqn:freeEnergy}
E_\eps [u] : = \begin{cases}
\int_{\Omega} \left(\frac{1}{\eps}W(x,u) + \frac{\eps}{2}|\nabla u|^2 \right) \de x & \text{ if } u\in H^1(\Omega),\\
+\infty & \text{ otherwise,}
\end{cases}
\end{equation}
where $W\in C^{{2}}(\Omega \times \R;[0,\infty))$ is a double-well function \emph{with moving wells}, namely satisfying
$$W(x,u) = 0  \text{ if and only if } u \in \{a(x),b(x)\}.$$
Here, we {suppose} $a,b \in C^2(\overline{\Omega})$ with $\inf\{b(x) - a(x): x\in \Omega\}=:\delta_{\rm sep} > 0$. Further, we {assume} that $W$ has quadratic growth near the wells, {in the sense that there are constants $C_1,C_2>0$ with}
\begin{equation}\label{eqn:quadratic growth}
C_1 \min\{|u -a(x)|, |u - b(x)|\}^2 \leq W(x,u) \leq C_2 \min\{|u -a(x)|, |u - b(x)|\}^2
\end{equation}
for all $u$ with $\min\{|u-a(x)|,|u-b(x)|\}<1,$ 
and {is $L^2$-coercive}, {meaning that there is $C>0$ such that}
\begin{equation}\label{eqn:L2 growth}
W(x,u) \geq \frac{1}{C}|u|^{{2}} - C.
\end{equation}
As remarked in the introduction (see \cite{Bouchitte}), the energies $E_\eps$ $\Gamma$-converge to the weighted perimeter functional
\begin{equation}\label{eqn:limitEnergy}
E[u] : = \begin{cases}
\int_{\Omega} \sigma \, \de |\nabla \chi_A|  & \text{ if }u = b\chi_A + a(1-\chi_A) \in BV(\Omega), \\
+\infty & \text{otherwise},
\end{cases}
\end{equation}
where, {due to $\delta_{\rm sep}>0$, $A$ is a set of finite perimeter in $\Omega$} and the surface energy density (or surface tension) is given by 
\begin{equation} \label{eq:surftenx}
    \sigma(x) : = \int_{a(x)}^{b(x)}\sqrt{2W(x,s)}\, \de s.
\end{equation}
{We remark that the main difference with respect to the (standard) result proven independently by Modica \cite{modica87} and Sternberg \cite{Sternberg1988} (see also \cite{ModicaMortola}) is that the surface tension is heterogeneous, rather than a constant, due to the spatially dependent potential.} 
 In the case {that} we wish to localize the energies, we will specify the domain $\Omega$ in the above energies by writing $E_\eps [u;\Omega]$ or $E[u;\Omega]$.

We will often make use of a normalized version of $W$, denoted by $ W_{\rm n}$, which adjusts the wells to be $0$ and $1$, while still having a spatially dependent energy landscape.
The \emph{normalized {double}-well function} is given by
$$W_{\rm n}(x,v) : = W(x, a(x) + \gamma (x)v),$$
where $$\gamma(x) : = b(x) - a(x) \geq \delta_{\rm sep}.$$
The function $W_{\rm n}$ gives rise to a normalized surface energy density defined as
\begin{equation}\label{eqn:sigman}
\sigma_{\rm n}(x) := \int_0^1 \sqrt{2 W_{\rm n}(x,s)}\, ds = \frac{\sigma(x)}{\gamma(x)}. 
\end{equation}
Using the quadratic growth \eqref{eqn:quadratic growth} and $\delta_{\rm sep}>0$, it is direct to show that there is $C>0$ such that $$1/C\leq \sigma(x) \leq C \quad \text{ for all $x\in \Omega$},$$ which also implies that $\sigma_{\rm n}$ is non-degenerate. 
We further suppose {that $\partial_x \sqrt{W_{\rm n}(x,v)}$ and $\partial_x^2 \sqrt{W_{\rm n}(x,v)}$ belong to $C(\overline{\Omega}\times \R; \R^{N})$ and $C(\overline{\Omega}\times \R; \R^{N\times N})$,} respectively, and that the potential satisfies the quantitative control
\begin{equation}\label{eqn:derivative control}
|\partial_x \sqrt{W_{\rm n}(x,v)} |\leq C\sqrt{W_{\rm n}(x,v)},
\end{equation}
for a constant $C>0$.

\begin{remark}[Control of the derivative]{\normalfont
To motivate the assumption (\ref{eqn:derivative control}), we note that if $W$ is quadratic near the wells---meaning, $C_1=C_2$ in (\ref{eqn:quadratic growth}) for $\min\{|u - a(x)|,|u - b(x)|\} \ll 1$---and has good growth at infinity, then this assumption is satisfied. {For example, we can verify bound \eqref{eqn:derivative control}, by noting that for $v$ near $0$ the normalized well is given by $\sqrt{W_{\rm n}(x,v)} = C_1\gamma |v|$.
Similarly,} one can directly verify that (\ref{eqn:derivative control}) holds for the canonical example $W(x,u) : = |u - a(x)|^2|u - b(x)|^2$.

}
\end{remark}

\begin{remark} {\normalfont
It is important to note that \eqref{eqn:derivative control} holds only for $W_{\rm n}$ and cannot hold for the function $W$ in the case of moving wells. Indeed, the Gronwall's inequality would imply that
$$W(x,a(y))\leq CW(y,a(y)) = 0$$ for some $C>0$ and $x,y\in \Omega$ such that the line between them is contained in $\Omega$, and hence the two wells of $W$ would be constant in each connected component of $\Omega$.}
\end{remark}

\section{Convergence of first variations and Gibbs--Thomson relation} \label{sec:firstVar}

{Recall that the Gibbs--Thomson relation for a homogeneous and isotropic two-phase system states that the chemical potential (i.e., the first variation of the diffuse interfacial energy) is proportional to the curvature of the interface between the phases (cf. \cite{GurtinConjecture}). In the heterogeneous setting, we derive the analogous result by first showing that  the chemical potential of the diffuse energy converges, as $\eps\to 0$, to the first variation of the weighted perimeter of the limit interface generally. Reinterpreting this result for mass-constrained minimizers of the {interfacial} energy \eqref{eqn:freeEnergy},
we recover a Gibbs--Thomson relation generalizing the result of \cite{LuckhausModica} to the heterogeneous energy \eqref{eqn:freeEnergy}.}

\begin{theorem}[Convergence of first variations]\label{theorem:firstvar}
{Let $\Omega\subset \R^N$ be a $C^2$-domain,} and let $u_\epsilon\in H^2(\Omega)$
{satisfy} $u_\epsilon \to u_A := {b\chi_A + a(1-\chi_A)} \in BV(\Omega)$ in $L^1(\Omega)$ {and} $E_\eps[u_\eps] \to E[u_A]$ {as $\eps\to 0$}. Define $v_\eps:=\frac{u_\eps-a}{\gamma}$. {Then for all $\Psi \in C^1(\overline{\Omega};\R^N)$ with $\Psi \cdot n_{ \Omega} = 0$ on $\partial \Omega$,} we have
\begin{equation}\label{eqn:GTthm}
\begin{aligned}
\nabla E_\epsilon [u_\eps](\gamma \nabla v_\eps \cdot \Psi) := &\int_{\Omega} \left(\frac{1}{\eps}\partial_{u} W(x,u_\epsilon) (\gamma \nabla v_\eps \cdot \Psi ) + \eps \nabla u_\eps \cdot \nabla(\gamma \nabla v_\eps \cdot \Psi) \right) \de x \\ 
& \to  \delta_{{(-\Psi)}} E[u_A] := -\int_{\Omega} \sigma ({\rm Id} - n_A\otimes n_A):\nabla \Psi \, \de |\nabla \chi_A|-\iO \nabla \sigma \cdot \Psi \, \de |\nabla \chi_A|
\end{aligned}
\end{equation}
as $\eps \to 0$, where 
$n_A$ is the inner normal to $A.$
\end{theorem}

We observe that the first term in \eqref{eqn:GTthm} corresponds to the {weak} formulation of the mean curvature of the reduced boundary $\partial^\ast A$. {In the case of minimizers, one expects more regularity for the limit interface and we may use the above result to connect the limit of chemical potentials to the curvature of the interface.}

\begin{corollary}[Gibbs--Thomson relation]\label{cor:GT}
 Let $u_\epsilon\in H^2(\Omega)$ minimize $E_\eps$ \eqref{eqn:freeEnergy} subject to the constraint $\dashint_\Omega u_\eps \, \de x = m$ with $m\in {(\dashint_\Omega a\,\de x,\dashint_\Omega b\,\de x)}$, so that \begin{equation}\label{eqn:EulerLagrangeWithConstraint}
 \lambda_\eps = \eps \Delta u_\eps - \frac1\eps\partial_u W(x, u_\eps) \quad \text{ for }x\in \Omega,
 \end{equation}where $
 \lambda_\eps\in \R$ is the Lagrange multiplier for the mass-constraint. \\ Suppose further that $u_\epsilon \to u_A := b\chi_A + a(1-\chi_A) \in BV(\Omega)$ in $L^1(\Omega)$, where  $A\subset \Omega$ has $C^2$-boundary.
 Then
 \begin{equation}\nonumber
\lim_{\eps \rightarrow 0} {\lambda_\eps} = {\frac{1}{\gamma}}{\div{(\sigma n_A)}}= {\frac{1}{\gamma}}{\left(-\sigma H_{A} + \nabla \sigma \cdot n_A\right)},
 \end{equation}
 where $H_A= - {\nabla} \cdot n_A$ is the mean curvature of $\partial A $. 
\end{corollary}

{We briefly comment on the hypotheses of the previous corollary.}
{\begin{remark}
We note that if $a=0$ and $b=1$, then the corollary holds in the heterogeneous setting without a priori assumptions on the regularity of $A$. Precisely, in dimension $N\leq 3$, using \cite{DePhilippisMaggi}, volume constrained minimizers $A$ of the weighted perimeter $E$ \eqref{eqn:limitEnergy} with regular surface tension have $C^2$-boundary. We suspect this regularity theory also holds more generally for functions $a$ and $b$ and sets $A$  satisfying the weighted constraint $\int_\Omega u_A \de x = m $ with $m \in (\dashint_\Omega a\,\de x,\dashint_\Omega b\,\de x)$, but a proof of this is beyond the scope of the paper.
\end{remark}}

We {also} note that there is a slight generalization of Theorem \ref{theorem:firstvar} {accounting for time-dependence}. This generalization will allow us to identify the surface velocity as the weighted mean curvature in our study of the sharp interface limit for the heterogeneous Allen--Cahn equation in Section \ref{sec:convAC}.

\begin{remark} \label{rmk:timeIntegratedFirstVar}
{Suppose} that for $T\in (0,\infty)$, there is a sequence $u_\eps \in L^2((0,T);H^2(\Omega))$ converging to a limit $u_A := {b\chi_A + a(1-\chi_A)} \in {L^1( (0,T);BV(\Omega))}$ in $L^1(\Omega\times (0,T))$  and is such that the energy convergence hypothesis
$$ \int_{0}^T E_\eps[u_\eps]\, \de t = \int_{0}^T E[u_A]\, \de t \quad \text{ as }\eps\to 0$$
holds. Then for $\Psi \in C([0,T];C^1(\overline{\Omega};\R^N))$ such that $\Psi\cdot n_{\Omega} = 0$ on $\partial \Omega\times (0,T)$, it holds that 
$$\int_0^T \nabla E_\epsilon [u_\eps](\gamma \nabla v_\eps \cdot \Psi)\, \de t \to  \int_{0}^T \delta_{{(-\Psi)}} E[u_A] \, \de t \quad \text{ as }\eps\to 0,$$
where we have used the notation from the theorem.
 As the proof of this generalization is the same as for Theorem \ref{theorem:firstvar}, but with numerous time integrals carried along, we only carry out the details for the static case of Theorem \ref{theorem:firstvar}.
\end{remark}

\subsection{Equipartition of energy}

{To prove the above results, we will often make use of the following intermediate result, which shows that the energy carried by the potential term and the gradient term is the same (even when localized).}

\begin{lemma}[Equipartition of energy]\label{lem:equipart}
Let $\{u_\epsilon\} \subset H^1(\Omega)$ and $A$ be a set of finite perimeter in $\Omega\subset \R^N$ with $u_\epsilon \to u_A = {b\chi_A + a(1-\chi_A)} \in BV(\Omega)$ in ${L^1(\Omega)}$. Supposing that $E_\eps[u_\eps] \to E[u] = \int_{\Omega}\sigma \, \de |\nabla \chi_A|$, the following relations hold:
\begin{align}
\lim_{\eps\to 0} \int_{\Omega}\left(\eps^{1/2}|\nabla u_\eps| -\frac{1}{\eps^{1/2}}\sqrt{2W(x,u_\eps)} \right)^2 \de x & = 0 \label{eqn:equipartition}\\
\frac{2}{\eps}W(x,u_\eps) \mathcal{L}^{N},\eps |\nabla u_\eps|^2\mathcal{L}^{N} ,\sqrt{2W(x,u_\eps)}|\nabla u_\eps|\mathcal{L}^{N}&\wkstr \sigma |\nabla \chi_A| \quad \text{ in }C(\overline{\Omega})^*. \label{eqn:energyConvergenceLocal}
\end{align}
\end{lemma}
{We note that similar statements hold when using normalized quantities; for instance, under the hypothesis of the lemma, it holds that $\eps |\gamma \nabla v_\eps|^2{\mathcal{L}^N} \wkstr \sigma |\nabla \chi_A|$ in $C(\overline{\Omega})^*$ where as before $v_\eps :=\frac{u_\eps -a}{\gamma}$. When clear, we shall use such replacements without further comment.}

\begin{proof}
We expand the square as 
\begin{equation}\label{eqn:expandSquare}
\frac{1}{2} \int_{\Omega}\left(\eps^{1/2}|\nabla u_\eps| -\frac{1}{\eps^{1/2}}\sqrt{2W(x,u_\eps)} \right)^2\, \de x = E_\eps[u_\eps] - \int_\Omega \sqrt{2W(x,u_\eps)}|\nabla u_\eps|\, \de x.
\end{equation}
But by \cite[equation (3.20)]{Bouchitte} and the proof of \cite[Theorem 4.1]{Bouchitte}, we have 
\begin{equation}\label{eqn:bouchLocal}
E[u] \leq \liminf_{\eps \to 0}\int_\Omega \sqrt{2W(x,u_\eps)}|\nabla u_\eps|\, \de x.
\end{equation} 
Consequently, taking the $\limsup$ of (\ref{eqn:expandSquare}) gives (\ref{eqn:equipartition}). 

For (\ref{eqn:energyConvergenceLocal}), we just prove $\sqrt{2W(x,u_\eps)}|\nabla u_\eps|\mathcal{L}^{N}\wkstr \sigma |\nabla \chi_A|$, as the others are similar. Let $\psi\in C(\overline{\Omega};[0,1])$ be a fixed test function. Noting that (\ref{eqn:bouchLocal}) holds for any open set in place of $\Omega$ {(with the energy on the left-hand side restricted to the set too)}, we use a standard Fubini-type trick along with Fatou's lemma to find
\begin{align*}
\iO \psi  \sigma \, \de |\nabla \chi_A| =& \int_0^\infty \int_{\{\psi > t\}}\sigma \, \de |\nabla \chi_A|\, \de t \\
\leq & \liminf_{{\eps\to 0}} \int_0^\infty \int_{\{\psi > t\}}\sqrt{2W(x,u_\eps)}|\nabla u_\eps|\, \de x\, \de t = \liminf_{{\eps\to 0}} \iO \psi\sqrt{2W(x,u_\eps)}|\nabla u_\eps|\, \de x.
\end{align*}
As the same argument works for $(1-\psi)$ and $\iO\sqrt{2W(x,u_\eps)}|\nabla u_\eps|\, \de x \to E[u]$ (by (\ref{eqn:expandSquare}) and (\ref{eqn:equipartition})), we find 
$$\limsup_{{\eps\to 0}} \iO \psi\sqrt{2W(x,u_\eps)}|\nabla u_\eps|\, \de x\leq \iO \psi  \sigma \, \de |\nabla \chi_A|. $$ Uniting the $\liminf$ and $\limsup$ inequalities gives us the desired convergence of measures.
\end{proof}

\subsection{Proof of Theorem \ref{theorem:firstvar} (Convergence of first variations)}

\begin{proof}
We begin by noting that $\nabla u_\eps = \nabla a  + v_\eps \nabla \gamma +  \gamma \nabla v_\epsilon.$ It follows by (\ref{eqn:L2 growth}) that $\int_{\Omega} \eps|v_\eps|^{{2}} \, \de x\leq C\eps$ and
$\int_\Omega \eps |\nabla v_\eps|^2 \,  \de x \leq C$. By H\"older's inequality {and Young's inequality, respectively, we see}
$\int_{\Omega}\eps |\nabla v_\eps|\, \de x \to 0$ {and $\int_{\Omega}\eps |v_\eps||\nabla v_\eps|\, \de x \to 0$} as $\eps \to 0.$

\textit{Step 1 (Reformulation of the first variation of $E_\eps$).}
Using Einstein summation notation, we begin by rearranging the second term in ${\nabla E_\epsilon (u_\eps)(\gamma \nabla v_\eps \cdot \Psi)}$. We compute as follows:
\begin{align*}
\int_{\Omega} \eps \nabla u_\epsilon \cdot \nabla (\gamma \nabla v_\eps \cdot \Psi) \, \de x &= \iO \eps \partial_i u_\eps \partial_i(\gamma \partial_j v_\eps \Psi_j) \de x\\
& = \iO \eps\left(\partial_i u_\eps \partial_i\gamma \partial_j v_\eps \Psi_j + \partial_i u_\eps \gamma \partial_i\partial_j v_\eps \Psi_j + \partial_i u_\eps \gamma \partial_j v_\eps \partial_i\Psi_j\right) \de x\\
& = \iO \eps\left(\gamma (\nabla v_\eps \otimes \nabla v_\eps) : (\nabla \gamma\otimes \Psi) + \gamma^2 (\nabla v_\eps \otimes \nabla v_\eps ): \nabla \Psi + \partial_i u_\eps \gamma \partial_i\partial_j v_\eps \Psi_j\right) \de x \\
& \quad\quad +o_{\eps \to 0}(1),
\end{align*} 
where in the last equality we swapped the second and third terms and picked up errors from replacing $\nabla u_\eps$ by $\gamma \nabla v_\eps$. 

We expand the last term in the above display:
\begin{align*}
\iO \eps \partial_i u_\eps \gamma \partial_i\partial_j v_\eps \Psi_j \, \de x = & \iO \eps\left(\gamma^2 \partial_i v_\eps \partial_i\partial_j v_\eps \Psi_j + v_\eps \partial_i \gamma \gamma \partial_i\partial_j v_\eps \Psi_j \right) \de x + o_{\eps \to 0}(1) \\
= & \iO \eps\left(\gamma^2 \partial_i v_\eps \partial_i\partial_j v_\eps \Psi_j - \gamma(\nabla v_\eps \otimes \nabla v_\eps) : (\nabla \gamma\otimes \Psi) \right) \de x + o_{\eps \to 0}(1),
\end{align*}
where we have used an integration by parts {(with $\partial_j$ to avoid boundary contributions)} in the second equality and picked up errors in both equalities using the comments preceding the start of Step 1. 

Now we use integration by parts on the $j$th derivative to rewrite the first term of the previous display:
$$\iO \eps \gamma^2 \partial_i v_\eps \partial_i\partial_j v_\eps \Psi_j \, \de x =  - \int_{\Omega} \eps \left(\gamma^2 \partial_j\partial_i v_\eps \partial_i v_\eps \Psi_j  +  |\nabla v|^2 2 \gamma \nabla \gamma \cdot \Psi +\gamma^2|\nabla v|^2 \nabla \cdot\Psi \right) \de x.$$
As the left-hand side and the first term on the right-hand side are the same, we find
$$\iO \eps \gamma^2 \partial_i v_\eps \partial_i\partial_j v_\eps \Psi_j \, \de x =  - \int_{\Omega}  \left( \eps |\nabla v|^2 \gamma \nabla \gamma \cdot \Psi +\frac{\eps}{2}\gamma^2|\nabla v|^2 \nabla \cdot\Psi \right) \de x. $$

Finally, synthesizing the above equations, we have
\begin{equation}\label{eqn:gradTerm}
\begin{aligned}
\int_{\Omega} & \eps \nabla u_\epsilon \cdot \nabla (\gamma \nabla v_\eps \cdot \Psi) \, \de x \\
&=  \- \iO ({\rm Id} - n_\eps \otimes n_\eps):\nabla \Psi\, {\de (\eps |\gamma \nabla v_\eps|^2) } - \int \frac{1}{\gamma} \nabla \gamma \cdot \Psi\,  \de (\eps |\gamma \nabla v_\eps|^2) + \iO \frac{\eps}{2} \gamma^2 |\nabla v_\eps|^2 \nabla\cdot \Psi \, \de x,
\end{aligned}
\end{equation}
where we define {the $\eps$-normal}
\begin{equation}\label{eqn:neps}
{n_\eps : = \begin{cases}\nabla v_\eps / |\nabla v_\eps| & \text{ if }\nabla v_\eps \neq 0 \\ 0 & \text{ otherwise.} \end{cases}}
\end{equation} We note that we write the first two integrals on the right-hand side with respect to the approximate surface measure {$\eps |\gamma \nabla v_\eps|^2 \mathcal{L}^N$}.

We now turn to the term coming from the potential $W$, i.e., the first term of ${\nabla E_\epsilon (u_\eps)(\gamma \nabla v_\eps \cdot \Psi)}$. For this, we note that 
$$ {\nabla (W_{\rm n}(x,v_\eps) )}= \partial_x W_{\rm n} (x,v_\eps) + \partial_{u}W(x,u_\eps)\gamma \nabla v_\eps. $$
Consequently we have
\begin{align*}
\int_{\Omega} \frac{1}{\eps}\partial_{u} W(x,u_\epsilon) (\gamma \nabla v_\eps \cdot \Psi ) \, \de x =& \iO \frac{1}{\eps}\left(\partial_x (W_{\rm n}(x,v_\eps)) - \partial_x W_{\rm n} (x,v_\eps)\right)\cdot \Psi\, \de x \\
=&  - \iO \frac{1}{\eps}W_{\rm n}(x,v_\eps) \nabla \cdot \Psi\, \de x  -\iO\frac{1}{\eps}\partial_x W_{\rm n} (x,v_\eps)\cdot \Psi \, \de x.
\end{align*}
Looking to the last term of (\ref{eqn:gradTerm}), we have 
\begin{equation}\label{eqn:chemRemains}
\begin{aligned}\int_{\Omega}&  \frac{1}{\eps}\partial_{u} W(x,u_\epsilon) (\gamma \nabla v_\eps \cdot \Psi ) \, \de x + \iO \frac{\eps}{2} \gamma^2 |\nabla v_\eps|^2 \nabla\cdot \Psi \, \de x  \\
=& \iO \left(\frac{\eps}{2} \gamma^2 |\nabla v_\eps|^2 - \frac{1}{\eps}W_{\rm n}(x,v_\eps)\right)\nabla \cdot \Psi \, \de x  -\iO\frac{1}{\eps}\partial_x W_{\rm n} (x,v_\eps)\cdot \Psi \, \de x \\
=& -\iO\frac{1}{\eps}\partial_x W_{\rm n} (x,v_\eps)\cdot \Psi \, \de x + o_{\eps \to 0}(1)
\end{aligned}
\end{equation}
by the equipartition of energy Lemma \ref{lem:equipart}.

\textit{Step 2 (Weak convergence of the variation in the potential).}
We now replace the term remaining on the right-hand side of (\ref{eqn:chemRemains}) by a more convenient term. In particular,
\begin{equation}\label{eqn:chemRemains2}
\begin{aligned}
\iO\frac{1}{\eps}\partial_x W_{\rm n} (x,v_\eps)\cdot \Psi \, \de x &= \iO \frac{2}{\eps}\sqrt{W_{\rm n}}(x,v_\eps)\partial_x \sqrt{W}(x,v_\eps) \cdot \Psi \, \de x \\
&= \iO \left(\partial_x \sqrt{2W_{\rm n}}(x,v_\eps)|\gamma \nabla v_\eps|\right)\cdot \Psi \, \de x \\
& \quad + \iO \frac{1}{\eps^{1/2}}\partial_x \sqrt{2W_{\rm n}}(x,v_\eps) \cdot \Psi \left(\frac{1}{\eps^{1/2}}\sqrt{2W_{\rm n}}(x,v_\eps) - \eps^{1/2}|\gamma \nabla v_\eps|\right) \, \de x \\
& = \iO \left(\partial_x \sqrt{2W_{\rm n}}(x,v_\eps)|\gamma \nabla v_\eps|\right)\cdot \Psi \, \de x + o_{\eps\to 0}(1),
\end{aligned}
\end{equation}
where we used the equipartition {Lemma \ref{lem:equipart}} and \eqref{eqn:derivative control} to recover an error term.

We claim that the vectorial measure \begin{equation}\label{eqn:weakConvergenceClaim}
\partial_x \sqrt{2W_{\rm n}}(x,v_\eps)|\gamma \nabla v_\eps|\mathcal{L}^{N} \wkstr \gamma \nabla \sigma_{\rm n} |\nabla \chi_A|  \quad {\text{ in }C(\overline{\Omega};\R^N)^*},
\end{equation}
where $\sigma_{\rm n}$ is the normalized surface energy in (\ref{eqn:sigman}). {We {postpone} the proof of this claim to conclude the theorem.}

\textit{Step 3 (Conclusion).} We now {combine} the equations we have derived and pass to the limit to conclude the proof of Theorem \ref{theorem:firstvar}. 
The first term in \eqref{eqn:gradTerm} can be handled with the same technique as in \cite{LuckhausModica}. Precisely, we apply \Cref{lem:equipart} to see that 
\begin{equation}\label{eqn:firstVarRewrite}
- \iO ({\rm Id} - n_\eps \otimes n_\eps):\nabla \Psi\, \de (\eps |\gamma \nabla v_\eps|^2) = - \iO ({\rm Id} - n_\eps \otimes n_\eps):\nabla \Psi\, \de (\sqrt{W_{\rm n}(x,v_\ep)}|\gamma\nabla v_\ep|) + o_{\eps \to 0}(1)
\end{equation}
Defining {the measures $\nu_\ep: = \sqrt{W_{\rm n}(x,v_\ep)}\gamma \nabla v_\ep \mathcal{L}^N$}, 
we have that ${\nu_\ep}$ converges in the weak-star sense and in total variation to the measure ${\sigma} \nabla \chi_A$ by Lemma \ref{lem:equipart}.
From this, we deduce
\begin{equation}\label{eqn:limit_2}
\begin{aligned}
- \iO ({\rm Id} - n_\eps \otimes n_\eps):\nabla \Psi\, \de {\nu_\eps}
\to -\int_{\Omega} \sigma ({\rm Id} - n_A\otimes n_A):\nabla \Psi \, \de  |\nabla \chi_A|.
\end{aligned}
\end{equation}
by the Reshetnyak continuity theorem \cite[Theorem 2.39]{AmbrosioFuscoPallara} applied to the sequence of measures $\nu_\eps$ with the {$1$-homogeneous functions $f(x,p) := {|p|}(I - \frac{p}{|p|} \otimes \frac{p}{|p|}):\nabla \Psi(x)$.}
{Altogether,} the equipartition Lemma \ref{lem:equipart} and equations (\ref{eqn:gradTerm}), (\ref{eqn:chemRemains}), (\ref{eqn:chemRemains2}), (\ref{eqn:weakConvergenceClaim}), \eqref{eqn:firstVarRewrite}, and \eqref{eqn:limit_2} give that
\begin{equation}\nonumber
\begin{aligned}
\nabla E_\epsilon (u_\eps)(\gamma \nabla v_\eps \cdot \Psi) \to &
{ -\int_{\Omega} \sigma ({\rm Id} - n_A\otimes n_A):\nabla \Psi \, \de  |\nabla \chi_A|} -\iO (\sigma_{\rm n}\nabla \gamma + \gamma\nabla \sigma_{\rm n}) \cdot \Psi \, \de |\nabla \chi_A|  \\
=& -\int_{\Omega} \sigma ({\rm Id} - n_A\otimes n_A):\nabla \Psi \, \de  |\nabla \chi_A|-\iO \nabla \sigma \cdot \Psi \, \de |\nabla \chi_A|,
\end{aligned}
\end{equation}
where in the last line we have used (\ref{eqn:sigman}). {Consequently, we have found \eqref{eqn:GTthm}.}

{\textit{Step 4 (Proof of \eqref{eqn:weakConvergenceClaim}).}}
{It remains to prove} (\ref{eqn:weakConvergenceClaim}). {We do this in two substeps: we initially} allow the measure to see the gradient direction, and then we show that the gradient norm can be recovered by freezing the ``$\eps$-normal" direction.

First, note that it suffices to prove (\ref{eqn:weakConvergenceClaim}) by considering test functions $\Psi \in C_c(\Omega;\R^N)$. To see this, given $\Psi \in C(\overline{\Omega};\R^N)$, let $\theta \in C_c(\Omega;[0,1])$ be a function such that $\theta \equiv 1$ on $\{{\rm dist}(x,\partial \Omega) > \eta\}$ for some $\eta>0.$ Then 
\begin{equation}\nonumber
\begin{aligned}
& \left|\int_\Omega \partial_x \sqrt{2W_{\rm n}}(x,v_\eps)|\gamma \nabla v_\eps| \cdot \Psi  \de x - \int_\Omega \nabla \sigma_{\rm n}\cdot \Psi \de|\nabla \chi_A| \right| \\
& \leq \left|\int_\Omega \partial_x \sqrt{2W_{\rm n}}(x,v_\eps)|\gamma \nabla v_\eps| \cdot (\theta \Psi)  \de x - \int_\Omega \nabla \sigma_{\rm n}\cdot (\theta \Psi) \de|\nabla \chi_A| \right| \\
& \quad \quad 
+ \|\nabla \sigma_{\rm n}\|_\infty|\nabla \chi_A|(\{{\rm dist}(x,\partial \Omega) < \eta\}) 
+ CE_\eps [u_\eps;\{{\rm dist}(x,\partial \Omega) < \eta\}],
\end{aligned}
\end{equation}
where we have applied \eqref{eqn:derivative control} to recover the last term. Taking the $\limsup$ as $\eps\to 0$, we can bound the right-hand side of the previous inequality by $$C|\nabla \chi_A|(\{{\rm dist}(x,\partial \Omega) \leq \eta\}),$$
where we have used the energy convergence hypothesis to ensure that we do not pick up any charge on $\partial \Omega$. Letting $\eta\to 0$, we recover (\ref{eqn:weakConvergenceClaim}). From now, we may assume $\Psi \in C_c(\Omega;\R^N)$.

\textit{Substep 4.1 (Weak convergence of the measure with full gradient).}
We study the measure $$\mu_\eps : = \partial_x \sqrt{2W_{\rm n}}(x,v_\eps) \otimes  \nabla v_\eps \mathcal{L}^N.$$ Define the geodesic distance for the normalized function $W_{\rm n}$ by $$d_{\rm n}(x,\xi) : = \int_{0}^\xi \sqrt{2W_{\rm n}}(x,t)\, \de t,$$ and note that $\partial_x d_{\rm n}(x,v_\eps) = \int_0^{v_\eps}\partial_x\sqrt{2W_{\rm n}}(x,t)\, \de t$ to find
$$\nabla\left(\partial_x d_{\rm n}(x,v_\eps)\right) = \partial_x^2 d_{\rm n}(x,v_\eps) + \partial_x\sqrt{2W_{\rm n}}(x,v_\eps) \otimes \nabla v_\eps.$$
Letting $\Phi \in C_c(\Omega; \R^{N\times N})$ we find
\begin{align}
\iO \Phi: \de \mu_\eps = & \iO {\nabla}\left(\partial_x d_{\rm n}(x,v_\eps)\right) : \Phi \, \de x -\iO \partial_x^2 d_{\rm n}(x,v_\eps):\Phi \, \de x \nonumber \\ 
= & -\iO \partial_x d_{\rm n}(x,v_\eps) \cdot (\nabla \cdot \Phi) \, \de x -\iO \partial_x^2 d_{\rm n}(x,v_\eps):\Phi \, \de x \nonumber \\ 
\xrightarrow[\varepsilon \to 0]{}& -\iO \partial_x d_{\rm n}(x,\chi_A) \cdot (\nabla \cdot \Phi) \, \de x -\iO \partial_x^2 d_{\rm n}(x,\chi_A):\Phi \, \de x. \label{eqn:limmueps}
\end{align}
Noting that $\partial_x d_{\rm n}(x,1) = \nabla \sigma_{\rm n}$ and $\partial_x^2 d_{\rm n}(x,1) = \nabla^2 \sigma_{\rm n}$ and using an integration by parts, we have that the right-hand side of (\ref{eqn:limmueps}) can be re-written as
\begin{equation}\nonumber
\begin{aligned}
-\iO \chi_A \nabla \sigma_{\rm n} \cdot (\nabla \cdot \Phi)\, \de x -\iO \chi_A \nabla^2 \sigma_{\rm n} :\Phi \, \de x  = \iO\nabla \sigma_{\rm n}\otimes n_A:\Phi \, \de |\nabla \chi_A|.
\end{aligned}
\end{equation}
It follows from the above equation and (\ref{eqn:limmueps}) that 
\begin{equation}\label{eqn:muepsConv}
\mu_\eps \wkstr \nabla \sigma_{\rm n}\otimes n_A |\nabla \chi_A| =: \mu \quad {\text{ in } C_c(\Omega;\R^{N\times N})^*}.
\end{equation}

\textit{Substep 4.2 (Freezing the normal to recover (\ref{eqn:weakConvergenceClaim})).}
{Recall the $\eps$-normal $n_\eps$ associated with the diffuse interface defined in \eqref{eqn:neps}.} Then we have
$$\iO \left( \partial_x \sqrt{2W_{\rm n}}(x,v_\eps)| \nabla v_\eps|\right)\cdot \Psi \, \de x = \iO \Psi \otimes n_\eps: \de \mu_\eps. $$
Introducing the function $\eta \in C(\Omega;\R^N)$ (which will approximate $n_A$ and thereby $n_\eps$), we have that
\begin{equation}\label{eqn:mueps1}
\iO \Psi \otimes n_\eps:\de \mu_\eps = \iO \Psi \otimes \eta:\de \mu_\eps  + \iO \Psi \otimes (n_\eps-\eta): \de \mu_\eps .
\end{equation}
The first term will converge as we pass $\eps\to 0$ by the weak convergence of $\mu_\eps$. To control the error arising in the second term, we argue as follows. First we apply H\"older's inequality and (\ref{eqn:derivative control}) to find that
\begin{equation}\label{eqn:hoelderApp}
\iO \Psi \otimes (n_\eps-\eta):\de \mu_\eps \leq C(\sup_{\eps}E_\eps[u_\eps]^{1/2})\left(\int_{{\rm supp}\Psi} |n_\eps-\eta|^2 \sqrt{2W_{\rm {n}}}(x,v_\eps)|\nabla v_\eps|\, \de x\right)^{1/2}.
\end{equation}
To understand the limit of the last term, we first note (analogous to before) that $$\nabla (d_{\rm n}(x,v_\eps)) = \partial_x d_{\rm n}(x,v_\eps) + \sqrt{2W_{\rm {n}}}(x,v_\eps)\nabla v_\eps,$$ and so letting $\psi$ be a $C^1_c(\Omega)$ cut-off function that is equal to $1$ on ${\rm supp}\Psi$, we have
\begin{align*}
\int_\Omega \sqrt{2W_{\rm n}}(x,v_\eps) \nabla v_\eps \cdot (\eta \psi)\, \de x =& -\iO d_{\rm n}(x,v_\eps)(\nabla \cdot (\eta\psi))\, \de x - \iO \partial_x d_{\rm n}(x,v_\eps)\cdot (\eta\psi) \, \de x \\
\xrightarrow[\varepsilon \to 0]{}& -\iO d_{\rm n}(x,\chi_A)(\nabla \cdot (\eta \psi)) \, \de x = \iO \psi(\eta\cdot n_A) \sigma_{\rm n} \, \de |\nabla \chi_A|.
\end{align*}
With this, we can estimate the last term of (\ref{eqn:hoelderApp}) as follows:
\begin{align*}
\int_{{\rm supp}\Psi} &|n_\eps-\eta|^2 \sqrt{2W_{\rm {n}}}(x,v_\eps)|\nabla v_\eps|\, \de x \\
\leq & \int_{\Omega} \psi|n_\eps-\eta|^2 \sqrt{2W_{\rm {n}}}(x,v_\eps)|\nabla v_\eps|\, \de x \\
= & \int_{\Omega} \psi(1 + \eta^2) \sqrt{2W_{\rm {n}}}(x,v_\eps)|\nabla v_\eps|\, \de x -2 \int_\Omega\sqrt{2W_{\rm n}}(x,v_\eps) \nabla v_\eps \cdot (\eta \psi)\, \de x \\
\xrightarrow[\varepsilon \to 0]{}&  \int_{\Omega} \psi(1 + \eta^2) \sigma_{\rm {n}}\, \de |\nabla \chi_A| -2 \int_\Omega \psi(\eta\cdot n_A) \sigma_{\rm n}\, \de |\nabla \chi_A| = \int_\Omega \psi|n_A-\eta|^2 \sigma_{\rm n}\, \de |\nabla \chi_A|,
\end{align*}
where we have once again used the equipartition Lemma \ref{lem:equipart}.

Recalling (\ref{eqn:muepsConv}), we now pass to the limit in (\ref{eqn:mueps1}) to find that
\begin{equation}\nonumber
\limsup_{\eps \to 0}\left|\iO \Psi \otimes n_\eps:\de \mu_\eps -  \iO \Psi \otimes n_A:\de \mu\right| \leq C \int_\Omega \psi|n_A-\eta|^2 \sigma_{\rm n}\, \de |\nabla \chi_A|.
\end{equation}
By density {in $L^2(|\nabla \chi_A|; \R^N)$}, the right hand side may be taken arbitrarily small, thereby showing that
$$\iO \left(\partial_x \sqrt{2W_{\rm n}}(x,v_\eps)| \nabla v_\eps|\right)\cdot \Psi \, \de x = \iO \Psi \otimes n_\eps:\de \mu_\eps \xrightarrow[\varepsilon \to 0]{}   \iO \Psi \otimes n_A:\de \mu =  \iO \nabla \sigma_{\rm {n}} \cdot \Psi \, \de  |\nabla \chi_A|,$$ which concludes the proof of the claim (\ref{eqn:weakConvergenceClaim}) as $\gamma$ is a continuous and positive function.

\end{proof}

\subsection{Proof of Corollary \ref{cor:GT} (Gibbs--Thomson relation)}

\begin{proof}
{We first conclude the corollary by assuming that the Lagrange multiplier $\lambda_\eps\in \R$ converges to a limit $\lambda_0\in \R$. In the course of this, we will deduce a quick way to prove that $\lambda_\eps$ converges. }

\textit{Step 1 ({Gibbs--Thomson relation assuming $\lambda_\eps\to \lambda_0$}).}
{We test the weak form of the equation \eqref{eqn:EulerLagrangeWithConstraint} with the function $\gamma \nabla v_\eps \cdot \Psi$ for given $\Psi \in C^1(\overline{\Omega};\R^N)$ satisfying} $\Psi \cdot n_{ \Omega} = 0$ on $\partial \Omega$ to find 
\begin{equation}\label{eqn:GT_ep}
   {\lambda_\eps} \int_\Omega{ \gamma \nabla v_\ep\cdot \Psi\;\de x}=-\int_{\Omega} \left(\frac{1}{\eps}\partial_{u} W(x,u_\epsilon) (\gamma \nabla v_\eps \cdot \Psi ) + \eps \nabla u_\eps \cdot \nabla(\gamma \nabla v_\eps \cdot \Psi) \right) \de x.
\end{equation}
For the left hand side, we can integrate by parts and use {that} $v_\ep \to \chi_A$ {in $L^1(\Omega)$} to get 
 $$\lim_{\eps\to 0} {\lambda_\ep}\int_\Omega{\gamma \nabla v_\ep\cdot \Psi\;\de x} = \lim_{\eps\to 0} - {\lambda_\ep}\int_\Omega{\div{(\gamma \Psi)}v_\ep\,\de x} = -{\lambda_0}\int_\Omega{\div{(\gamma \Psi)}\chi_A\,\de x} = {\lambda_0} \int_\Omega{\gamma \Psi\cdot n_A\,\de |\nabla \chi_A|}.$$
 For the right hand side of \eqref{eqn:GT_ep}, we can compute the limit by applying Theorem \ref{theorem:firstvar} and integrating by parts using the fact that $A $ has $C^2$ boundary. {Precisely, passing to the limit in both sides of \eqref{eqn:GT_ep},} we obtain
 \begin{align*}
   {\lambda_0} \int_\Omega{\gamma \Psi\cdot n_A\,\de |\nabla \chi_A|} & =  \int_{\Omega} \sigma ({\rm Id} - n_A\otimes n_A):\nabla \Psi \, \de |\nabla \chi_A|+\iO \nabla \sigma \cdot \Psi \, \de |\nabla \chi_A|  \\
   & =  \iO  \left(-\sigma H_A +\nabla \sigma \cdot n_A\right) \Psi \cdot n_A\, \de |\nabla \chi_A| 
 \end{align*}
 where $H_A= - {\nabla \cdot} n_A$ is the mean curvature of the boundary of $A$. 
 {As $\Psi$ is arbitrary,} it follows that
$ \gamma {\lambda_0} = - \sigma H_A +  \nabla \sigma \cdot n_A $ along the boundary of $A$, and we recover the desired result.

{\textit{Step 2 (Convergence of $\lambda_\eps$).} Choose $\Psi \in C^1_c(\Omega;\R^N)$ such that $\int_\Omega{\gamma \Psi\cdot n_A\,\de |\nabla \chi_A|} \neq 0$, which is possible since $A$ has finite perimeter and nonempty reduced boundary (due to the mass-constraint). Equation \eqref{eqn:GT_ep} can be rewritten as
\begin{equation}\nonumber
\lambda_\eps = \frac{-\int_{\Omega} \left(\frac{1}{\eps}\partial_{u} W(x,u_\epsilon) (\gamma \nabla v_\eps \cdot \Psi ) + \eps \nabla u_\eps \cdot \nabla(\gamma \nabla v_\eps \cdot \Psi) \right) \de x}{\int_\Omega{ \gamma \nabla v_\ep\cdot \Psi\;\de x}}.
\end{equation}
As in Step 1, Theorem \ref{theorem:firstvar} guarantees the numerator converges, and the denominator converges to $\int_\Omega{\gamma \Psi\cdot n_A\,\de |\nabla \chi_A|}$. Thus, $\lambda_\eps$ converges to some $\lambda_0\in \R$ as desired.}
\end{proof}

\section{Convergence analysis of the gradient flow}\label{sec:convAC}

{The purpose of this section is show that, as $\eps \to 0$, solutions of the heterogeneous Allen--Cahn equation \eqref{eqn:AChetero} converge to a solution of weighted mean curvature flow \eqref{eqn:wMCFstrong}.
As we will consider} relatively general initial conditions, we {make use of} suitable weak solution concepts for both the diffuse interface equation  and its target curvature {flow.} For the Allen--Cahn equation we introduce the following solution concept. As the proof for existence of solutions {follows from} a standard minimizing movements scheme, we defer this result to Appendix {\ref{sec:ACexist}}.

\begin{definition}[Weak solution of the heterogeneous Allen--Cahn equation]\label{def:weakACsoln}
Let $N \geq 2$ with $\Omega\subset \R^N$ a $C^2$-domain, $W\in C^1(\overline{\Omega}\times \R)$, and consider a finite time horizon $T \in (0, \infty)$. 
For $\eps>0 $ and $u_{0,\eps}\in H^1(\Omega)\cap L^\infty(\Omega)$, we say a function
$$u_\eps \in H^{1}((0,T);L^2(\Omega))\cap L^2((0,T);H^2(\Omega))\cap L^\infty (\Omega\times (0,T)) $$ 
is a weak solution of the heterogeneous Allen--Cahn equation \eqref{eqn:AChetero} with initial condition $u_{0,\eps}$ if the following are satisfied:
\begin{enumerate}
\item \textit{(Initial condition).} $u_\eps(\cdot, 0) = u_{0,\eps}$.
\item \textit{(Evolution law).} For all $\phi \in C^1(\overline{\Omega})$ and $t\in (0,T)$,
$$ \int_{\Omega} \partial_t u_\eps(x,t) \phi \, \de x  = - \int_{\Omega} \Big(\nabla u_\eps(x,t)\cdot \nabla 
\phi + \frac{1}{\eps^2}\partial_u W(x,u_\eps(x,t))\phi \Big) \de x .$$
\item \textit{(Optimal dissipation equality)} For every $T'\in [0,T]$ 
\begin{equation}\label{eqn.dissipACen}
 E_\eps[u_\eps(\cdot,T')] + \int_0^{T'}\int_{\Omega} \eps |\partial_t u_\eps|^2\de x \, \de t = E_\eps [u_{0,\eps}].
\end{equation}
\end{enumerate}
\end{definition}

Our weak solution concept for weighted mean curvature flow relies on sets of finite perimeter and, in particular, their gradient, which were discussed in Subsection {\ref{subsec:notation}}. We will also rely on disintegrated measures (cf. Young measures), as introduced in \cite[Section 2.5]{AmbrosioFuscoPallara}.
\begin{definition}[{$BV$} solutions to weighted mean curvature flow]\label{def:distribsol}
Let $N \geq 2$ with $\Omega\subset \R^N$ a $C^2$-domain and consider a finite time horizon $T \in (0, \infty)$. Suppose $\sigma\in C^1(\overline{\Omega} ; (0,\infty)) $. For a set of finite perimeter $A_0 \subset \Omega$, we say a time-parametrized family of sets $A(t) \subset \R^N$ with finite perimeter, for $t \in [0, T]$, is a distributional (or BV) solution to weighted mean curvature flow \eqref{eqn:wMCFstrong} with initial condition $A_0$ if:
\begin{enumerate}
	\item (Existence of normal velocity and initial condition). There exists a $|\nabla \chi_{A(t)}| \otimes \mathcal{L}^1\llcorner (0,T)$-measurable function $V$
	such that
	\begin{equation}\label{eqn:VL2}
		\int_0^T \int_{\Omega} \sigma | V|^2 \, \mathrm{d}|\nabla \chi_{A(t)}| \,\mathrm{d}t < \infty 
	\end{equation}
and $V (\cdot,  t)$ is the normal velocity of $\partial^*A(t)$ in the direction $n_{A(t)} := \frac{\nabla \chi_{A(t)}}{|\nabla \chi_{A(t)}|}$ in the sense that for almost every $T' \in (0,T)$ and all $\zeta \in C^\infty_c(\overline{\Omega} \times [0,T))$, it holds that
\begin{equation}\label{eqn:transport}
	\int_{A(T')}  \zeta(\cdot, T') \, \mathrm{d}x - 	\int_{A_0}  \zeta(\cdot, 0) \, \mathrm{d}x  = \int_0^{T'} \int_{A(t)}  \partial_t \zeta \, \mathrm{d}x\, \mathrm{d}t - \int_0^{T'}
	\int_{\Omega}  V  \zeta  \, \mathrm{d}|\nabla \chi_{A(t)}| \, \mathrm{d}t.
\end{equation}
	\item (Motion law). For almost every $t \in  (0, T)$ and any test vector field $\Psi \in C^1 (\overline{\Omega};\R^N)$ with $\Psi\cdot n_\Omega = 0 $ on $\partial \Omega$, it holds that
	\begin{align}\nonumber
	& \int_{\Omega} \sigma V (\Psi\cdot  n_{ A(t)})   \, \mathrm{d} |\nabla \chi_{A(t)}| \\
	&\ \  =  - \int_{\Omega} \sigma (\operatorname{Id} -  n_{A(t)} \otimes  n_{A(t)} ): \nabla \Psi \, \mathrm{d}|\nabla \chi_{A(t)}|
		- 	 \int_{\Omega} \nabla \sigma \cdot \Psi \, \mathrm{d}|\nabla \chi_{A(t)}|. \label{eqn:motionlaw} 
	\end{align}
	\item (Optimal energy dissipation rate). For almost every $T' \in (0,T)$, we have
	\begin{equation}\label{eqn:optimaldissip}
		E[u_{{A(T')}}]+ \int_{0}^{T'} \int_{\Omega} \sigma |V|^2\, \mathrm{d}|\nabla \chi_{A(t)}|\, \mathrm{d}t \leq E[u_{A_0}],
	\end{equation}
where $u_A := b \chi_A + a (1-\chi_A)$ and $E[u_A]$ is defined in \eqref{eqn:limitEnergy}.
\end{enumerate}
\end{definition}
Note that $|\nabla \chi_{A(t)}| \otimes \mathcal{L}^1\llcorner (0,T)$-measurability is the same as being Borel-measurable, since the disintegrated measure is a Radon measure (see \cite[Section 2.5]{AmbrosioFuscoPallara}). We briefly remark that {the first item} of the solution concept is satisfied by generic interfaces evolving with square-integrable velocities. It is only in the motion law {\eqref{eqn:motionlaw}} that the velocity is specified to be the {weighted} mean curvature. {Additionally, {\eqref{eqn:motionlaw}} encodes $90^{\circ}$ contact angles for the interface along the domain boundary; other contact angles can be accounted for using a boundary energy term as in \cite{HenselLaux-contact}.} Finally, the dissipation law {\eqref{eqn:optimaldissip}} is a defining characteristic of the evolution that can be used to recover uniqueness for flows (see Theorem \ref{theo:weakstrong}). In fact, as in \cite{HenselLaux-varifold}, the dissipation can also be used to encode the motion law; we do not write our solution this way for the sake of clarity.

We turn to the main result of this section and show that under a suitable energy convergence hypothesis, solutions of the heterogeneous Allen–Cahn equation converge (up to a subsequence) to a distributional solution of weighted {mean curvature flow}. 

\begin{theorem}[Conditional convergence of the gradient flows]\label{theo:convAC}
 Let $N\geq 2$ with $\Omega\subset \R^N$ a $C^2$-domain, $T \in (0, \infty)$ be a finite time horizon, and let $W$ satisfy the assumptions within Subsection \ref{subsec:setting}. Consider a sequence $(u_{0,\eps})_{\eps >0}$ of initial conditions with 
 \begin{equation}\label{eqn:uniformControlIC}
 \sup_{\eps>0} \|u_{0,\eps}\|_{L^\infty(\Omega)} < \infty 
 \end{equation}
  such that there exists a set of finite perimeter $A_0$ satisfying
  \begin{equation}\label{eqn:convE0}
 \begin{aligned}
 	u_{0,\eps} &\rightarrow b \chi_{A_0} + a (1 - \chi_{A_0}) = : u_{A_0}  && \quad \text{ in } L^1(\Omega) \text{ as } \eps \rightarrow 0,\nonumber
 	\\
 	E_\eps[u_{0,\eps}] &\rightarrow E[u_{A_0}]  && \quad \text{ as } \eps \rightarrow 0. 
 \end{aligned}
 \end{equation}
 Let $u_\eps $ be the associated weak solution of the heterogeneous Allen-Cahn problem with initial condition $u_{0,\eps}$ as in Definition \eqref{def:weakACsoln} such that $\sup_{\eps>0}\|u_\eps\|_{L^\infty(\Omega\times (0,T))} <\infty$. Then the following holds:
 \begin{enumerate}
 \item \textit{(Compactness).}
 There exists a subsequence $\eps \rightarrow 0$ and a time-parametrized family of sets $A(t)$ with finite perimeter, for $t \in [0, T]$, such that
 \begin{equation}\label{eqn:convergenced}
 	u_\eps  \rightarrow   b \chi_{A} +a (1 - \chi_{A}) =: u_{A}  \quad \text{ strongly in } L^1(\Omega \times (0,T)) \text{ as } \eps \rightarrow 0,
 \end{equation}
where the identification $\chi_{A}(\cdot,t): = \chi_{A(t)}$ is such that $\chi_A \in C^{1/2}([0,T]; L^1(\Omega)) \cap BV(\Omega \times (0,T))$.

\item \textit{(Limit evolution).}
Assuming the energy convergence hypothesis
\begin{equation}\label{eqn:hpenergyconv}
	\lim_{\eps\to 0} \int_{0}^{T} E_\eps (u_\eps (\cdot, t)) \, \mathrm{d}t =  \int_{0}^{T} E ({u_{A (t)}}) \, \mathrm{d}t ,
\end{equation}
the flow $t\mapsto A(t)$  is a distributional solution to weighted mean curvature flow with initial condition $A_0$ as in Definition \ref{def:distribsol}.
\end{enumerate}
\end{theorem}

Note, to simplify the proof of compactness we {consider} solutions of the heterogeneous Allen--Cahn equation that satisfy the bound $\sup_{\eps>0}\|u_\eps\|_{L^\infty(\Omega\times (0,T))} <\infty$. Theorem {\ref{thm:ACexist}} provides such solutions as a direct consequence of a maximum principle. One could alternatively relax the $L^\infty$ control on the initial condition in \eqref{eqn:uniformControlIC}, and impose polynomial growth on $W$ to obtain the solutions of the heterogeneous Allen--Cahn equation.

\begin{proof}[Proof {of Theorem \ref{theo:convAC}}]
In the following, we use the geodesic distances $d$ and $d_{\rm n}$ associated to the functions $W$ and $W_{\rm n},$ respectively, defined by 
\begin{equation}\nonumber
d(x,u) : = \int_{a(x)}^u \sqrt{2 W(x,s)}\, \de s \quad \text{ and } \quad d_{\rm n}(x,v) : = \int_{0}^v \sqrt{2 W_{\rm n}(x,s)}\, \de s.
\end{equation}
With $v_\eps : = \frac{u_\eps - a}{\gamma}$, we further define $d^\eps_{\rm n}(x) : = d_{\rm n}(x,v_\eps{(x)})$.

\textit{Step 1 (Compactness).} Note that the claimed convergence will follow from the convergence of $v_\eps$ to an appropriate characteristic function. To obtain compactness for $v_\eps$, we will show that $d^\eps_{\rm n}$ is uniformly bounded in $BV(\Omega\times (0,T))$, and then pull this information back to $v_{\eps}$. First, 
we have that $\|v_\eps\|_{L^\infty(\Omega\times (0,T))}<C $ for some $C>0$ and all $\eps>0.$ Due to the continuity of $d_{\rm n}$, we have also that $\|d^\eps_{\rm n}\|_{L^\infty(\Omega\times (0,T))}\leq C$ and, in particular, {$d_{\rm n}^\eps$} is uniformly bounded in {$L^1(\Omega\times (0,T))$} for all $\eps>0.$ 

For control on the gradients, we note that 
\begin{equation}\label{eqn:derivednRel}
\partial_t d^\eps_{\rm n} = \sqrt{2W(\cdot, u_\eps)} \partial_t u_\eps/\gamma \quad  \text{ and }\quad \nabla d^\eps_{\rm n} = \sqrt{2W(\cdot, u_\eps)} \nabla v_\eps  + \partial_x d_{\rm n}(\cdot,v_\eps)
\end{equation} 
Applying Young's inequality and the optimal dissipation relation \eqref{eqn.dissipACen}, we see that $$\| \partial_t d^\eps_{\rm n}  \|_{L^1(\Omega\times (0,T))} \leq C \sup_{\eps>0} E_\eps(u_{\eps,0})< \infty $$ and, 
{since $\partial_x d_{\rm n}$ is continuous with $v_\eps$ uniformly bounded,}
that 
 \begin{equation}\nonumber 
 \| \nabla  d^\eps_{\rm n}  \|_{L^1(\Omega\times (0,T))} \leq {C\Big(\sup_{\eps>0} E_\eps(u_{\eps,0}) + \| {\partial_x d_{\rm n}(\cdot,v_\eps)}  \|_{L^1(\Omega\times (0,T))}\Big)}  < \infty.
 \end{equation} 
 Compactness for uniformly bounded sequences in $BV$  \cite[Theorem 3.23]{AmbrosioFuscoPallara} shows that $d^{\eps}_{\rm n}$ converges in $L^1(\Omega\times (0,T))$ to a limit $d^0_{\rm n}\in BV(\Omega\times (0,T))$. Let $d^{-1}_{\rm n}(x,\cdot)$ be the inverse in the second input, which is well-defined (because $W_{\rm n}\geq 0$ with equality only at two points) and continuous in the second input. Thus possibly taking a subsequence in $\eps$, we have $$ v_\eps(x)  = d^{-1}_{\rm n}(x,d^\eps_{\rm n}(x))\to d^{-1}_{\rm n}(x,d^0_{\rm n}(x)) =: v_0(x) \quad \text{pointwise a.e. in $\Omega\times (0,T)$,}$$
 and therefore in $L^1(\Omega\times (0,T))$ up to a subsequence {(not relabeled)}. Note that measurability of $v_0$ follows since it is the limit of measurable functions, and we do not need the continuity of $d^{-1}_{\rm n}$ in space.

 To see that $v_\eps$ is the characteristic function of  a set, note that the dissipation relation \eqref{eqn.dissipACen} and {the convergence} $E_\eps [u_{0,\eps}] \to E[u_{A_0}]$ {as $\eps\to 0$} imply
 \begin{equation}\label{eqn:energyBoundallEps}
 \sup_{\eps>0 } \sup_{t\in [0,T]}E_\eps[u_\eps{(\cdot, t)}]\leq \sup_{\eps>0 } E_\eps[u_{0,\eps}]<\infty.
 \end{equation}
 Then energy bound \eqref{eqn:energyBoundallEps} and Fatou's lemma give
 $\int_{0}^T\int_{\Omega}W_{\rm n}(x,v_0(x))\, \de x \, \de t = 0,$ so that $v_0(x,t) = \chi_{A}(x,t)$ for some set $A\subset \Omega\times (0,T).$ Mapping this forward under $d_{\rm n}$, we see that $d^0_{\rm n} = \sigma_{\rm n}\chi_A$. Since $\sigma_{\rm n}$ is strictly bounded away from $0,$ we may use the coarea formula to conclude that $A$ is a set of finite perimeter in $\Omega\times (0,T)$, or equivalently $\chi_A\in BV(\Omega\times (0,T))$, as claimed. Here $A(t)$ is identified by the relation $\chi_{A(t)}: = \chi_{A}(\cdot, t)$.
 
The continuity in time of $A(t)$ in $L^1(\Omega)$ follows from the first relation in \eqref{eqn:derivednRel} for $\partial_t d^\eps_{\rm n}$. Precisely, considering $t<t'$ and using H\"older's inequality, the energy bound \eqref{eqn:energyBoundallEps}, and the dissipation equality \eqref{eqn.dissipACen}, we have
\begin{equation}\nonumber
\begin{aligned}
\int_{\Omega} |d^{\eps}_{\rm n}(\cdot, t) - d^{\eps}_{\rm n}(\cdot, t')| \,\de x   \leq \int_{t}^{t'}\int_{\Omega} |\partial_ t d^{\eps}_{\rm n}| \, \de x \, \de s & \leq C\left((t'-t)\sup_{\eps>0} E_\eps(u_{\eps,0})\right)^{1/2} \left(\int_{t}^{t'}\int_{\Omega} \eps |\partial_t u_\eps|^2 \, \de x\, \de s\right)^{1/2} \\
& \leq C (t' - t)^{1/2}\sup_{\eps>0} E_\eps(u_{\eps,0}),
\end{aligned}
\end{equation}
from which we pass $\eps\to 0$ to find
$\int_\Omega \sigma_{\rm n}|\chi_{A(t)} - \chi_{A(t')}| \, \de x \leq C (t' - t)^{1/2},$
thereby concluding the result since $\sigma_{\rm n}$ is bounded away from $0$.

 \textit{Step 2 (Existence of solutions).} We prove that solutions of weighted mean curvature flow exist under the energy convergence hypothesis \eqref{eqn:hpenergyconv}.
	
\textit{Substep 2.1 (Existence of the velocity $V$).} Arguing as in the proof of \cite[Lemma 7]{HenselLaux-contact}, one can show that $A(t)$ satisfies (1) of Definition \ref{def:distribsol}. We include the essential details for the reader.

Taking $\phi \in C_c^1 (\overline{\Omega}\times (0,T))$, we consider $\partial_t \chi_{A}$ as a distribution, use the compactness {of Step 1}, gradient relations \eqref{eqn:derivednRel}, the dissipation equality \eqref{eqn.dissipACen}, and the energy convergence hypothesis \eqref{eqn:hpenergyconv} to estimate
\begin{equation}\label{eqn:L2density}
\begin{aligned}
\langle \sigma \partial_t \chi_A,\phi \rangle & = - \lim_{\eps \to 0} \int_0^T \int_{\Omega}\sigma \frac{d^\eps_{\rm n}}{\sigma_{\rm n}}\partial_t\phi \, \de x \, \de t \\
& =  \lim_{\eps \to 0} \int_0^T \int_{\Omega}\sqrt{2W(x, u_\eps)} \partial_t u_\eps  \phi \, \de x \, \de t. \\
&\leq \liminf_{\eps \to 0} \left(\int_0^T \int_{\Omega} \eps |\partial_t u_\eps|^2 \, \de x \, \de t\right)^{1/2} \left(\int_0^T \int_{\Omega}  |\phi|^2 \frac{W(x, u_\eps)}{\eps} \, \de x \, \de t\right)^{1/2}\\
& \leq  \sup_{\eps>0 } E_\eps[u_{0,\eps}]^{1/2} \left(\int_0^T \int_{\Omega}  \sigma |\phi|^2 \, \de |\nabla \chi_{A(t)}| \, \de t\right)^{1/2}.
\end{aligned}
\end{equation}
Taking the supremum over $\phi$ supported in open sets with $|\phi|\leq 1$, {\eqref{eqn:L2density} implies} that the measure $\partial_t \chi_A$ is absolutely continuous with respect to $|\nabla \chi_{A(t)}|\otimes \mathcal{L}^1\llcorner (0,T)$. Defining $$V:=  - \frac{\de (\partial_t \chi_A)}{\de ( |\nabla \chi_{A(t)}|\otimes \mathcal{L}^1\llcorner (0,T) )},$$ the transport equation \eqref{eqn:transport} follows for $\zeta \in C_c^1(\overline{\Omega}\times (0,T))$; one can include the temporal boundary conditions using the fact that $\chi_{A(t)}$ is continuous in $L^1(\Omega).$ The bound \eqref{eqn:L2density} shows that integration with $V$ defines a bounded linear operator over $L^2(\sigma |\nabla \chi_{A(t)}|\otimes \mathcal{L}^1\llcorner (0,T)),$ thereby implying the claimed square integrability of $V$ in \eqref{eqn:VL2}.

\textit{Substep 2.2 (Motion law).}
Adopting the notation of Theorem \ref{theorem:firstvar}, (2) of Definition \ref{def:weakACsoln} shows that for $\Psi \in C^1_c(\overline{\Omega}\times(0,T);\R^N)$ with $\Psi\cdot n_{\Omega} = 0$ on $\partial \Omega\times (0,T)$, the weak solution $u_\eps$ satisfies
\begin{equation}\label{eqn:evoPreLim}
\int_{0}^{T} \int_{\Omega}(\eps ( \gamma\nabla v_\eps \cdot \Psi) \partial_t u_\eps )\, \de x \de t  = -\int_0^T \nabla E_\epsilon [u_\eps](\gamma \nabla v_\eps \cdot \Psi) \, \de t
\end{equation} 
 We further {claim} 
\begin{equation}\label{eqn:leftPass}
	\lim_{\eps\to 0} \left( \int_{0}^{T} \int_{\Omega}\eps ( \gamma\nabla v_\eps \cdot \Psi) \partial_t u_\eps \, \de x \de t \right) = -\int_{0}^{T} \int_{\Omega}  (\Psi \cdot n_{A(t)}) \sigma V \, \de |\nabla \chi_{A(t)}| \de t 
\end{equation}
for	all such test vector fields $\Psi.$
With this, we can conclude the proof of \eqref{eqn:motionlaw} by passing to the limit in \eqref{eqn:evoPreLim} using \eqref{eqn:leftPass} for the left side and Theorem \ref{theorem:firstvar} with Remark \ref{rmk:timeIntegratedFirstVar} for the right side. This gives the time integrated version of \eqref{eqn:motionlaw}, which can subsequently be localized {in time}.

We now prove the claim \eqref{eqn:leftPass}. 
For any $\eta \in C^1([0, T]; C(\overline{\Omega}; \R^N)) $ with $|\eta| \leq  1$ (approximating $n_{A(t)}$), adding zero twice, we can write
\begin{align}
\int_{0}^{T} \int_{\Omega} (\eps ( \gamma\nabla v_\eps \cdot \Psi) \partial_t u_\eps )\, \de x \, \de t 
&=  \int_{0}^{T} \int_{\Omega} \Psi \cdot (n_\eps - \eta ) \sqrt{\eps} \gamma|\nabla v_\eps|  \sqrt{\eps} \gamma \partial_t v_\eps \, \de x \, \de t \nonumber \\
&\quad + \int_{0}^{T} \int_{\Omega} (\Psi \cdot \eta) \Big(\sqrt{\eps} \gamma|\nabla v_\eps|  - \frac{1}{\sqrt{\eps}} \sqrt{2W_{\rm n}(x,v_\eps)}\Big) \sqrt{\eps} \gamma \partial_t v_\eps \, \de x \, \de t  \nonumber \\
&\quad + \int_{0}^{T} \int_{\Omega}  (\Psi \cdot \eta ) \gamma \partial_t d_{\rm n}(x,v_\eps) \, \de x \, \de t . \label{eqn:needt}
\end{align}
To pass to the limit in the last term, we compute
\begin{align*}
& \lim_{\eps\to 0}	\int_{0}^{T} \int_{\Omega}  (\Psi \cdot \eta )\gamma \partial_t d_{\rm n}(x,v_\eps) \, \de x \de t \\
& = 	- \lim_{\eps\to 0} \bigg(\int_{0}^{T} \int_{\Omega}  \partial_t (\Psi \cdot \eta ) d(x,u_\eps) \, \de x \de t +\int_{\Omega}   (\Psi \cdot \eta )(\cdot, 0) d(x,u_{\eps,0}) \, \de x \de t \bigg) \\
	&=   - \int_{0}^{T} \int_{\Omega}  \partial_t (\Psi \cdot \eta ) \sigma \chi_A  \, \de x \de t - \int_{\Omega}   (\Psi \cdot \eta )(\cdot, 0) \sigma \chi_A(\cdot, 0) \, \de x \de t \\
	&= {-} \int_{0}^{T} \int_{\Omega}  (\Psi \cdot \eta ) \sigma V  \, \de |\nabla \chi_{A(t)}| \de t. 
\end{align*}

Similar to the argument in Substep 4.2 of Theorem \ref{theorem:firstvar}, as $\eps\to 0$, one can estimate the first two terms on the right-hand side of \eqref{eqn:needt} by 
$$ C\left(\int_0^T \int_\Omega |n_{A(t)} - \eta|^2 \, \de |\nabla \chi_{A(t)}| \, \de t\right)^{1/2}$$
using the energy convergence hypothesis \eqref{eqn:hpenergyconv}, the time-integrated version of {Lemma \ref{lem:equipart}} (cf. Remark \ref{rmk:timeIntegratedFirstVar}), and the dissipation inequality \eqref{eqn.dissipACen}.
By density, we can let $\eta \to n_{A(t)}$ in $L^2(|\nabla \chi_{A(t)}| \otimes \mathcal{L}^1\llcorner (0,T))$ to conclude the proof the claim.  

\textit{Substep 2.3 (Optimal energy dissipation).} 
Note that \eqref{eqn:L2density} implies
\begin{equation}\nonumber
\begin{aligned}
-\int_0^T\int_\Omega \sigma V & \phi \, \de |\nabla \chi_{A(t)}| \, \de t = \langle \sigma \partial_t \chi_A,\phi \rangle \\
&\leq  \left( \liminf_{\eps \to 0} \int_0^T \int_{\Omega} \eps |\partial_t u_\eps|^2 \, \de x \, \de t\right)^{1/2}  \left(\int_0^T \int_{\Omega}  \sigma |\phi|^2 \, \de |\nabla \chi_{A(t)}| \, \de t\right)^{1/2},
\end{aligned}
\end{equation}
and letting $\phi$ approximate $-V$ we recover
\begin{equation}\label{eqn:dissLSC} \int_0^T \int_{\Omega}  \sigma |V|^2 \, \de |\nabla \chi_{A(t)}| \, \de t \leq \liminf_{\eps \to 0} \int_0^T \int_{\Omega} \eps |\partial_t u_\eps|^2 \, \de x \, \de t.
\end{equation}
The inequality \eqref{eqn:optimaldissip} then follows from \eqref{eqn.dissipACen}, \eqref{eqn:convE0}, \eqref{eqn:dissLSC}, and the fact that $E_\eps[u_\eps(\cdot,t)] \to E[u_{{A(t)}}]$ for almost every $t \in (0,T)$ due to \eqref{eqn:hpenergyconv}.
\end{proof}

\section{Weak-strong uniqueness for weighted mean curvature flow}
\label{sec:weakStrong}

Relying on {convergence of} the first variations ({see} Theorem \ref{theorem:firstvar}), we proved {that weak} solutions of the heterogeneous Allen–Cahn equation converge, up to a subsequence, to a $BV$
solution {of} weighted mean curvature flow in the sense of Definition \ref{def:distribsol} ({see} Theorem \ref{theo:convAC}). 
{In this section, we prove that our notion of distributional solution satisfies a weak-strong uniqueness principle.
When a strong solution of the curvature flow exists, this principle uniquely identifies the limit of solutions of the Allen--Cahn equation as the strong solution of the curvature flow, so that convergence holds for the entire sequence.}

{Our proof of weak-strong uniqueness for $BV$ solutions of weighted mean curvature flow follows the strategy used in {\cite{FischerHenselLauxSimon}}
for (standard) mean curvature flow, relying on the relative energy technique. 
For simplicity, we will consider an evolving surface in the full space $\R^d$, rather than in $\Omega,$ to focus on the effect of spatial heterogeneity.
To extend the analysis in a bounded domain $\Omega$, one needs to also account for contact of the evolving interface with the domain boundary. 
For this, one should be able to exploit ideas already developed in the previous works \cite{HenselLaux-contact,HenselMoser,HenselMarveggio} on the relative energy method for curvature driven flows with constant contact angle conditions.}

{We} introduce the notion of gradient flow calibration {for weighted mean curvature flow.}
\begin{definition}[Gradient flow calibration for weighted mean curvature flow] \label{def:gradflowcal}
    Let $\mathcal{A}(t)$, {for} $t \in [0,T]$, be a {time-parametrized} family of open subsets with smooth boundary $\partial \mathcal{A} (t)$. 
    {Fix the surface tension $\sigma\in C^{1,1}(\R^N ; (0,\infty))$ such that $0<\delta \leq \sigma \leq C_\delta$ for some constants $\delta, C_\delta>0$. }
    Let $\xi, B : \R^N \times [0,T] \rightarrow \R^N$ and let $\vartheta: \R^N \times [0,T] \rightarrow \R$. We call the tuple $(\xi, B, \vartheta)$ a {calibration} for {weighted} mean curvature flow if the following statements hold true.
    \begin{itemize}
        \item[i)] (Regularity). The tuple $(\xi, B, \vartheta)$ satisfy
        \begin{align}
      &      \xi \in C^{{2}} ([0,T]; C_c(\R^N)) \cap C([0,T]; C^1_c(\R^N)) , \quad B \in C([0,T]; C^1_c(\R^N)), \\
&\vartheta \in C^{{1}}(\R^d \times [0,T]) \cap L^\infty (\R^N \times [0,T]).
        \end{align}
        \item[ii)] (Extension of the normal vector field).
         The vector field $\xi$ extends the interior unit normal vector field of $\mathcal{A}(t)$, namely 
\begin{align} \label{eqn:xiconstraint}
    \xi(\cdot, t) = n_{\mathcal{A}(t)} \quad \text{ on } \partial \mathcal{A}(t),
\end{align}
and there exists a constant $c >0$ such that 
\begin{align}\label{eq:lenghtxi}
   | \xi(\cdot, t) |\leq \max\{0, 1-c \dist^2(\cdot, \partial  \mathcal{A}(t)) \},
\end{align}
for all $t \in [0,T]$.
         \item[iii)] (Extension of the normal velocity). The vector field $B$ extends the normal velocity of $\mathcal{A}(t)$, namely 
         \begin{align}\label{eqn:Bconstraint}
    B(\cdot, t) = \mathcal{V}n_{\mathcal{A}(t)} \quad  \text{ on } \partial \mathcal{A}(t).
\end{align}
         \item[iv)]  (Sign condition and coercivity for the transported {mass}).  The function $\vartheta$ represents a transported {mass} subject to
the sign conditions
\begin{align}
    \vartheta(\cdot, t ) &>0 \quad  \text{ in } \mathcal{A}(t), \\
    \vartheta(\cdot, t ) &< 0 \quad \text{ in  } \R^N \setminus \overline{\mathcal{A}(t)}, \\
     \vartheta(\cdot, t ) &= 0 \quad \text{ on  } \partial {\mathcal{A}(t)},
\end{align}
for all $t \in [0,T]$. Moreover, there exists a constant $C>0$ such that
\begin{align}
\min\{\dist (\cdot, \partial \mathcal{A}(t) ) , 1\} \leq C |\vartheta (\cdot, t)|
\end{align}
for all $t \in [0,T]$.
         \item[v)] (Approximate evolution equations). The tuple $(\xi, B, \vartheta)$ evolves in time according to 
         \begin{align}
	\partial_t \xi + (B \cdot \nabla ) \xi + (\nabla B )^\mathsf{T} \xi &=O(\dist(\cdot, \partial {\mathcal{A}(t)} )),  \label{eqn:evxi}
	\\
		\partial_t |\xi|^2 + (B \cdot \nabla ) |\xi |^2  &=O(\dist^2(\cdot, \partial {\mathcal{A}(t)} )), \label{eqn:evxi2}
\\ \label{eqn:evvartheta}
	 \partial_t \vartheta + (B \cdot \nabla ) \vartheta  &=O(\dist(\cdot, \partial {\mathcal{A}(t)} )),
\end{align}
for all $t \in [0,T]$.
          \item[vi)] (Geometric equation). For all $t \in [0,T]$, {it holds that}
          \begin{align}\label{eq:geomB}
-	\nabla \cdot \xi -  \nabla \log \sigma \cdot \xi =  B \cdot \xi + O(\dist (\cdot, \partial {\mathcal{A}(t)} )).
\end{align}
    \end{itemize}
\end{definition}
{We note that, at odds with convention, we call $\vartheta$ a `mass' since we reserve `weight' for $\sigma.$}

{We comment that the class of calibrated flows includes strong solutions of the curvature flow.} {Precisely, consider a} time-parametrized family of bounded open sets $\mathcal{A}(t) \subset \R^N$, {for} $t \in [0,T]$, with $\partial \mathcal{A} (t)$ smoothly evolving by weighted mean curvature flow, {in the sense that} the {surface} normal velocity $\mathcal{V}(\cdot, t)$ of $\partial \mathcal{A} (t)$  is given by
\begin{equation}\nonumber
	\mathcal V(\cdot, t) = - \frac{1}{\sigma}\div{(\sigma n_{ \mathcal{A}(t)} )} = H_{\partial \mathcal{A}(t)}  - \nabla \log \sigma  \cdot n_{ \mathcal{A}(t)} ,
\end{equation}
where $ H_{{\mathcal{A}(t)}}  $ denotes the mean curvature of $\partial \mathcal{A} (t)$ and $\sigma$ the heterogeneous surface {tension}.
{We briefly write down a calibration for the smooth evolution and refer to \cite[Section 4.1-4.2]{FischerHensel} (see also \cite[Lemma 5.5]{LauxStinsonUllrich22}) for  the technical details verifying that it is a calibration as in Definition \eqref{def:gradflowcal}.}
For each $t \in [0, T ]$, let $P_{\partial \mathcal{A}(t)}: \R^N \rightarrow \partial \mathcal{A}(t)$ be the nearest point projection {onto} 
$\partial \mathcal{A}(t)$ and let ${r}>0$ be such that $P_{\partial \mathcal{A}(t)}$ is smooth in $\mathcal U _{{r}} (t)$, a tubular neighborhood {around $\partial \mathcal A(t)$ with} width ${r}$, for {every} $t \in [0,T]$. 
One can define the vector field $\xi$ as
\begin{align} \nonumber
	\xi(\cdot ,t) := g( \dist(\cdot , \partial \mathcal{A}(t) )) n_{ \mathcal A(t)}(P_{\partial \mathcal{A}(t)} \cdot ) \quad  \text{ in } {\R^N}, 
\end{align}
where $g$ is a smooth {cut-off} function satisfying $g(0)=1$, $g(s) \leq \max \{1-{c} s^2, 0\}$ for some $c>1/r^2$ so that $g$ is zero outside of the tubular neighborhood of $\partial \mathcal A(t)$.
Similarly, one can define 
\begin{align} 
    B(\cdot ,t) &:= \mathcal V(P_{\partial \mathcal{A}(t)} \cdot, t ) {\xi(\cdot ,t)}\quad  \text{ in } {\R^N}, \nonumber \\
	\vartheta(\cdot ,t) &:= \tau( \sdist(\cdot , \partial \mathcal{A}(t) ))  \quad  \text{ in } {\R^N}, \nonumber
\end{align} 
where $\tau$ is a smooth and non-decreasing truncation of the identity such that $\tau(s) = s$ for $|s| \leq {r}/2$ and $\tau(s)={r}\operatorname{sign} (s)$ for $|s| \geq {r}$.

Given a {$BV$} solution {$A(t)$} in the sense of Definition \ref{def:distribsol} and a {calibrated} evolution {$\mathcal{A}(t)$} as in Definition \ref{def:gradflowcal}, {both defined} up to a common finite time horizon $T \in (0, \infty)$, we introduce {the} relative energy functional
\begin{equation}\label{eqn:relen}
	E_{\text{rel}}[u_A|u_{\mathcal{A}} ](t):= \int_{\R^N} \sigma (1 - n_{A(t)}\cdot \xi) \, \de |\nabla \chi_{A(t)}|
\end{equation}
where $u_{\mathcal{A}} := b \chi_{\mathcal A} + a (1-\chi_{\mathcal A})$, and the bulk energy functional
\begin{equation}\label{eqn:bulken}
	E_{\text{bulk}}[u_A|u_{\mathcal{A}} ] (t):= \int_{\R^N} \sigma (\chi_{\mathcal A(t)} - \chi_{A(t)}) \vartheta \, \de x . 
\end{equation}
{Note that both integrands are positive everywhere.}
These two error functionals measure the difference between a given $BV$ solution and a {calibrated} evolution. {Precisely, we prove that they satisfy stability estimates leading to a weak-strong uniqueness principle: }

\begin{theorem}[Weak-strong uniqueness for weighted {mean curvature flow}] \label{theo:weakstrong}
 Let $N \geq 2$. Consider $T \in (0, \infty)$ and let $\mathcal A (t)$, $t \in [0,  T]$, be a calibrated gradient flow in the sense of Definition \ref{def:gradflowcal}. Let $A(t)$, $t \in [0, T^*
]$ , $T^* \geq T$, be a $BV$ solution to {weighted mean curvature flow} in the sense of Definition \ref{def:distribsol}.
Then there exists a constant $C > 0$ depending only on $(\xi, B, \vartheta)$ such that for almost every $T' \in (0,T)$ 
\begin{align}
	E_{\mathrm{rel}}[u_A|u_{\mathcal{A}}](T') &\leq E_{\mathrm{rel}}[u_A|u_{\mathcal{A}} ](0) + C \int_{0}^{T'} E_{\mathrm{rel}}[u_A|u_{\mathcal{A}} ](t) \, \de t , \label{eqn:relenestimate}\\
	  E_{\mathrm{bulk}}[u_A|u_{\mathcal{A}}](T') &\leq E_{\mathrm{rel}}[u_A|u_{\mathcal{A}}](0) + E_{\mathrm{bulk}}[u_A|u_{\mathcal{A}} ](0) + C \int_{0}^{T'} (E_{\mathrm{rel}}[u_A|u_{\mathcal{A}}](t)+ E_{\mathrm{bulk}}[u_A|u_{\mathcal{A}} ](t)) \, \de t , \label{eqn:bulkenestimate}
\end{align}
where $u_A := b \chi_{A} + a (1-\chi_{A})$ and $u_{\mathcal{A}} := b \chi_{\mathcal A} + a (1-\chi_{\mathcal A})$.
In particular,  we obtain a weak-strong uniqueness principle, i.e., if ${ A (0)} = {\mathcal{A}{(0)}} $ up to a set of zero Lebesgue measure, then ${ A (t)} = {{\mathcal{A} (t)}} $ up to a set of zero Lebesgue measure for a.e. $t \in (0, T)$.
\end{theorem}

\subsection{Intermediate result: Coercivity properties}
First, we collect some basic coercivity properties of the relative energy functional \eqref{eqn:relen}. 
\begin{lemma}
	Under the assumptions of Theorem \ref{theo:weakstrong}, there exists a constant $C>0$ such that the relative energy functional \eqref{eqn:relen} satisfies
	\begin{align}
		\int_{\R^N}  \frac12 \sigma |n_{A(t)}- \xi |^2 \,\de |\nabla \chi_{A(t)}| \leq C E_{\operatorname{rel}}[u_A|u_{\mathcal{A}} ](t),  \label{eqn:coern}\\
  \int_{\R^N} \sigma \min\{\dist^2 (\cdot, \partial \mathcal{A}(t) ) , 1\}\,\de |\nabla \chi_{A(t)}| \leq C  E_{\operatorname{rel}}[u_A|u_{\mathcal{A}} ](t), \label{eqn:coerdist} \\
			\int_{\R^N} \sigma \vartheta^2 \,\de |\nabla \chi_{A(t)}| \leq C  E_{\operatorname{rel}}[u_A|u_{\mathcal{A}} ](t), \label{eqn:coervartheta}
	\end{align}
 for all $t \in [0,T]$.
\end{lemma}
\begin{proof}
    Using the identity $2(1- \xi \cdot n_{A(t)} )= | n_{A(t)}- \xi |^2 + 1- |\xi|^2  $, where the second term is nonnegative due to \eqref{eq:lenghtxi}, the first two coercivity properties \eqref{eqn:coern}-\eqref{eqn:coerdist} directly follows. The third coercivity property \eqref{eqn:coervartheta} can then be deduced due to the Lipschitz continuity of the transported {mass} function $\vartheta$. 
\end{proof}

\subsection{Proof of Theorem \ref{theo:weakstrong}: Weak-strong uniqueness for $BV$ solutions to weighted MCF}
{We now establish a weak-strong uniqueness principle for {$BV$} solutions of weighted mean curvature flow.
As mentioned, our proof follows the lines of that for (standard) mean curvature flow \cite{FischerHenselLauxSimon} (see also \cite{HenselLaux-contact,Laux-volume}).} However, the presence of a spatially dependent surface tension {introduces additional terms that must be accounted for in} the {computations.}

\begin{proof}
We divide the proof into {three} parts. First we prove quantitative stability for the relative energy functional \eqref{eqn:relen}, then for the bulk energy functional \eqref{eqn:bulken}. {Finally,} the weak-strong uniqueness principle follows from the stability estimates \eqref{eqn:relenestimate} and \eqref{eqn:bulkenestimate} by means of a Gronwall argument.

\textit{a) Proof of the relative energy estimate \eqref{eqn:relenestimate}:}
	
	\textit{Step 1 (First manipulations).} Using \eqref{eqn:transport} and \eqref{eqn:optimaldissip}, we compute
	\begin{align*}
		&	E_{\text{rel}}[u_A|u_{\mathcal{A}} ](T') - 	E_{\text{rel}}[u_A|u_{\mathcal{A}} ](0) \\
   &\leq - \int_{0}^{T'}\int_{\R^N} \sigma |V|^2 \, \de |\nabla \chi_{A(t)}|  \de t  - \int_{0}^{T'}\int_{\R^N}   V (\sigma \nabla \cdot \xi +  \nabla \sigma \cdot \xi ) \, \de |\nabla \chi_{A(t)}|  \de t \\
			&\quad  -  \int_{0}^{T'}\int_{\R^N} \sigma \partial_t \xi \cdot n_{A(t)}  \, \de |\nabla \chi_{A(t)}|  \de t .
	\end{align*}
Adding \eqref{eqn:motionlaw} {integrated in time} with $B$ as a test function, we obtain
	\begin{align}
	&  E_{\text{rel}}[u_A|u_{\mathcal{A}} ](T') - 	E_{\text{rel}}[u_A|u_{\mathcal{A}} ](0)   \notag \\
&	\leq - \int_{0}^{T'}\int_{\R^N} \sigma |V|^2 \, \de |\nabla \chi_{A(t)}| \de t   	- \int_{0}^{T'}\int_{\R^N}   V (\sigma \nabla \cdot \xi +  \nabla \sigma \cdot \xi ) \, \de |\nabla \chi_{A(t)}| \de t \notag \\
& \quad + 	\int_{0}^{T'} \int_{\R^N} B \cdot  n_{A(t)}  \sigma V  \, \de |\nabla \chi_{A(t)}| \de t
+ \int_{0}^{T'} \int_{\R^N} \nabla \sigma \cdot B \, \de |\nabla \chi_{A(t)}|  \de t   
			 \notag \\
&\quad + \int_{0}^{T'} \int_{\R^N} \sigma (\operatorname{Id} -  n_{A(t)} \otimes  n_{A(t)} ): \nabla B \, \de |\nabla \chi_{A(t)}|  \de t
-  \int_{\R^N} \sigma \partial_t \xi \cdot n_{A(t)}  \, \de |\nabla \chi_{A(t)}|  \de t
.  \label{eqn:evErel}
\end{align}

	\textit{Step 2 (Dissipation estimates).} Consider the first three terms on the right hand side in the inequality \eqref{eqn:evErel}. By adding zeros to complete the squares, we obtain
	\begin{align*}
	&	- \int_{0}^{T'} \int_{\R^N} \sigma |V|^2 \, \de |\nabla \chi_{A(t)}| \de t  -  \int_{0}^{T'}\int_{\R^N}   V (\sigma \nabla \cdot \xi +  \nabla \sigma \cdot \xi ) \, \de |\nabla \chi_{A(t)}| \de t  	 \\
 &+ 	\int_{0}^{T'} \int_{\R^N} B \cdot  n_{A(t)}  \sigma V  \, {d}|\nabla \chi_{A(t)}|  \de t\\
	&= -  \int_{0}^{T'}\int_{\R^N} \sigma {\frac12} |V +  \nabla \cdot \xi +  \nabla \log \sigma \cdot \xi |^2 \, \de |\nabla \chi_{A(t)}| \de t
	-   \int_{0}^{T'}\int_{\R^N} \sigma {\frac12} |V  n_{A(t)}   - B|^2  \, \de |\nabla \chi_{A(t)}| \de t\\
	& \quad  +  \int_{0}^{T'}\int_{\R^N} \sigma \frac12  |\nabla \cdot \xi +  \nabla \log \sigma \cdot \xi |^2 \, \de |\nabla \chi_{A(t)}| \de t  +  \int_{0}^{T'}\int_{\R^N} \sigma \frac12 |B|^2   \, \de |\nabla \chi_{A(t)}| \de t. 
	\end{align*}
	
	\textit{Step 3 (Further manipulations).} Consider the last line on the right{-}hand side in the inequality \eqref{eqn:evErel}. 
 {Thinking of \eqref{eqn:evxi} and \eqref{eqn:evxi2}, we add zeros and rearrange the terms to obtain}
	\begin{align}
	&-   \int_{0}^{T'}\int_{\R^N} \sigma \partial_t \xi \cdot n_{A(t)}  \, \de |\nabla \chi_{A(t)}|  \de t +	\int_{0}^{T'} \int_{\R^N} \sigma (\operatorname{Id} -  n_{A(t)} \otimes  n_{A(t)} ): \nabla B \, \de |\nabla \chi_{A(t)}|  \de t \notag \\ 
	&	= -   \int_{0}^{T'}\int_{\R^N} \sigma( \partial_t \xi + (B \cdot \nabla ) \xi + (\nabla B)^\mathsf{T}  \xi )\cdot (n_{A(t)}- \xi   )\, \de |\nabla \chi_{A(t)}| \de t 
\notag	\\
	 & \quad  -   \int_{0}^{T'}\int_{\R^N} \sigma( \partial_t \xi + (B \cdot \nabla ) \xi  )\cdot \xi   \, \de |\nabla \chi_{A(t)}|  \de t  + \int_{0}^{T'} \int_{\R^N} \sigma (\nabla \cdot B) (1 - n_{A(t)}  \cdot \xi ) \, \de |\nabla \chi_{A(t)}|  \de t  \notag \\
	& \quad
	  {+}   \int_{0}^{T'}\int_{\R^N} \sigma( (\nabla B)^\mathsf{T}  \xi )\cdot ({n_{A(t)} - \xi}  )  \, \de |\nabla \chi_{A(t)}|  \de t 
	   -\int_{0}^{T'} \int_{\R^N} \sigma (  n_{A(t)} \otimes  n_{A(t)} ): \nabla B \, \de |\nabla \chi_{A(t)}|  \de t \notag \\
	  & \quad 
	 +  \int_{0}^{T'} \int_{\R^N} \sigma(  (B \cdot \nabla ) \xi )\cdot n_{A(t)}  \, \de {|\nabla \chi_{A(t)}|} \de t  + \int_{0}^{T'} \int_{\R^N} \sigma (\nabla \cdot B) (  n_{A(t)}  \cdot \xi ) \, \de |\nabla \chi_{A(t)}|  \de t   . \label{eq:step3ws}
	\end{align}
Note that, using the evolution equations \eqref{eqn:evxi} and \eqref{eqn:evxi2} together with the coercivity properties \eqref{eqn:coern}-\eqref{eqn:coerdist}, the first three terms in the right hand side of \eqref{eq:step3ws} can be bounded by the {(time integrated)} relative energy functional {as in \eqref{eqn:relenestimate}}.
Therefore, we need to control the two last lines in \eqref{eq:step3ws}. To process the last line, we use the equality $\nabla \cdot \nabla \cdot (\xi \otimes \sigma B) =\nabla \cdot \nabla \cdot (\sigma B\otimes \xi)$ 
and Gauss’ theorem, obtaining
\begin{align*}
	0 & = - \int_{\R^N}  \chi_{A(t)} \nabla \cdot (\nabla \cdot (\sigma B \otimes \xi - \xi \otimes \sigma B)) \, \de x \\
	&= \int_{\R^N} n_{A(t)}  \cdot (\nabla \cdot (  \sigma B \otimes \xi - \xi \otimes  \sigma B)) \, \de |\nabla \chi_{A(t)}| \\
	& =  \int_{\R^N} \sigma n_{A(t)}  \cdot (\nabla \cdot (  B \otimes \xi - \xi \otimes  B)) \, \de |\nabla \chi_{A(t)}| \\
	& \quad +  \int_{\R^N} (n_{A(t)}  \cdot B) (\nabla \sigma \cdot \xi )  \, \de |\nabla \chi_{A(t)}| 
		-   \int_{\R^N} (n_{A(t)}  \cdot \xi) (\nabla \sigma \cdot B ) \,  \de |\nabla \chi_{A(t)}| , 
\end{align*}
whence we can deduce
\begin{align*}
	& \int_{0}^{T'} \int_{\R^N} \sigma(  (B \cdot \nabla ) \xi )\cdot n_{A(t)}  \, \de |\nabla \chi_{A(t)}|  \de t  + \int_{0}^{T'} \int_{\R^N} \sigma (\nabla \cdot B) (  n_{A(t)}  \cdot \xi ) \, \de |\nabla \chi_{A(t)}|  \de t \\
	&=  \int_{0}^{T'}\int_{\R^N} \sigma(  (\xi \cdot \nabla ) B)\cdot n_{A(t)}  \, \de |\nabla \chi_{A(t)}|  \de t  + \int_{0}^{T'} \int_{\R^N} \sigma (\nabla \cdot \xi ) (  n_{A(t)}  \cdot B) \, \de |\nabla \chi_{A(t)}|  \de t \\
	& \quad +  \int_{0}^{T'} \int_{\R^N} (n_{A(t)}  \cdot B) (\nabla \sigma \cdot \xi )  \, \de |\nabla \chi_{A(t)}| \de t 
	-   \int_{0}^{T'}\int_{\R^N} (n_{A(t)}  \cdot \xi) (\nabla \sigma \cdot B ) \, \de |\nabla \chi_{A(t)}| \de t . 
\end{align*}
Hence, combining the terms from the {last} two lines of \eqref{eq:step3ws}, we obtain
\begin{align*}
	& 
	-\int_{0}^{T'} \int_{\R^N} \sigma (  (n_{A(t)} - \xi )\otimes  (n_{A(t)} - \xi )): \nabla B \, \de |\nabla \chi_{A(t)}| \de t\\
&  +  \int_{0}^{T'}\int_{\R^N} (n_{A(t)}  \cdot B) (\nabla \sigma \cdot \xi  + \sigma (\nabla \cdot \xi )) \,  \de |\nabla \chi_{A(t)}|  \de t 
-  \int_{0}^{T'} \int_{\R^N} (n_{A(t)}  \cdot \xi) (\nabla \sigma \cdot B ) \, \de |\nabla \chi_{A(t)}|  \de t ,
\end{align*}
where the first term can be bounded by the relative energy functional \eqref{eqn:relen} due to the coercivity property \eqref{eqn:coern}.

\textit{Step 4 (Final estimates).}
From the previous steps we deduce
\begingroup
\allowdisplaybreaks
\begin{align} \notag
	&E_{\text{rel}}[u_A|u_{\mathcal{A}} ](T') - 	E_{\text{rel}}[u_A|u_{\mathcal{A}} ](0)\notag  \\
&+	\int_{0}^{T'}\int_{\R^N} \sigma {\frac12}|V +  \nabla \cdot \xi +  \nabla \log \sigma \cdot \xi |^2 \, \de |\nabla \chi_{A(t)}| \de t
	+ \int_{0}^{T'}\int_{\R^N} \sigma {\frac12}|V  n_{A(t)}   - {(B\cdot \xi) \xi } |^2  \, \de |\nabla \chi_{A(t)}| \de t \notag \\
	&\leq  C\int_{0}^{T'}  E_{\text{rel}}[u_A|u_{\mathcal{A}} ](t)\, \de t {\, + 	\int_{0}^{T'} \int_{\R^N} (B- (B \cdot \xi ) \xi ) \cdot  n_{A(t)}  \sigma V  \, {d}|\nabla \chi_{A(t)}|  \de t }  \notag \\
	&\quad  +\int_{0}^{T'} \int_{\R^N} \sigma \frac12  |\nabla \cdot \xi +  \nabla \log \sigma \cdot \xi |^2 \, \de |\nabla \chi_{A(t)}| \de t  + \int_{0}^{T'} \int_{\R^N} \sigma \frac12 |{(B \cdot \xi) \xi }|^2   \, \de |\nabla \chi_{A(t)}| \de t \notag \\
	& \quad  + \int_{0}^{T'} \int_{{\R^N}} (n_{A(t)}  \cdot B) (\nabla \sigma \cdot \xi  + \sigma (\nabla \cdot \xi )) \,  \mathrm{d}|\nabla \chi_{A(t)}| \de t  + \int_{0}^{T'} \int_{\R^N} (\nabla \sigma \cdot B)  (1- n_{A(t)}  \cdot \xi)\, \mathrm{d}|\nabla \chi_{A(t)}|  \de t   
	\notag \\
		&\leq  C\int_{0}^{T'}  E_{\text{rel}}  [u_A|u_{\mathcal{A}} ](t) \, \de t \notag \\
	&\quad  +\int_{0}^{T'} \int_{\R^N} \sigma \frac12  |\nabla \cdot \xi +  \nabla \log \sigma \cdot \xi + B \cdot \xi |^2 \, \de |\nabla \chi_{A(t)}| \de t 
+\int_{0}^{T'} \int_{\R^N} \sigma \frac12 {|B\cdot \xi|^2(|\xi|^2 - 1)}   \, \de |\nabla \chi_{A(t)}| \de t \notag \\
	& \quad  + {\int_{0}^{T'} \int_{\R^N} \sigma  (n_{A(t)}\cdot (\operatorname{Id} - \xi \otimes \xi ) B)(V + \nabla \log \sigma \cdot \xi  + \nabla \cdot \xi )  \, \de |\nabla \chi_{A(t)}| } \de t \notag \\
	&\quad 
    {-  \int_{0}^{T'} \int_{\R^N} (1- {n_{{A(t)}} \cdot \xi})( \xi \cdot B) (\nabla \sigma \cdot \xi  + \sigma (\nabla \cdot \xi ))  \, \de |\nabla \chi_{A(t)}|  \de t}  \notag \\
	&\quad 
     + \int_{0}^{T'} \int_{{\R^N}} (\nabla \sigma \cdot B)  (1- n_{A(t)}  \cdot \xi)\, \mathrm{d}|\nabla \chi_{A(t)}| \de t   . \label{eq:relenineq} 
\end{align}
\endgroup
On the right{-}hand side of \eqref{eq:relenineq}, the second line and the last two lines can be bounded by the relative energy functional \eqref{eqn:relen} due to the geometric equation \eqref{eq:geomB}, the length condition \eqref{eq:lenghtxi},
 together with the coercivity property \eqref{eqn:coerdist},
  and the regularity of $\sigma$, $\xi$, and $B$. 
In order to estimate the third line, we use Young's inequality and {absorb the subsequent integrand $\delta |V +  \nabla \cdot \xi +  \nabla \log \sigma \cdot \xi |^2$, with $0<\delta\ll 1$, into} the first dissipation term on the left{-}hand side. {For the other term coming from Young's inequality, we compute}
\begin{align*}
&|n_{A(t)}\cdot (\operatorname{Id} - \xi \otimes \xi ) |^2 
\leq {2}|n_{A(t)} - \xi|^2 + 2 ( 1-n_{A(t)}  \cdot \xi ) ,
\end{align*}
where we used the length condition 
 \eqref{eq:lenghtxi}. 
Hence, the remaining {term (coming from Young's inequality) on} the right-hand side of \eqref{eq:relenineq} is bounded by the {time integrated} relative energy functional \eqref{eqn:relen} due to the coercivity property \eqref{eqn:coern}.
{Altogether, this proves \eqref{eqn:relenestimate}.}

\textit{b) Proof of the bulk energy estimate \eqref{eqn:bulkenestimate}.}

Using the equation \eqref{eqn:transport} with $\vartheta$ as test function, the analogue of \eqref{eqn:transport} for the smooth flow $\mathcal{A}(t)$, and the fact that $\vartheta = 0$ on $\partial \mathcal A(t)$, we deduce
\begin{align}
		E_{\text{bulk}}[u_A|u_{\mathcal{A}} ](T') - E_{\text{bulk}}[u_A|u_{\mathcal{A}} ](0) =  \int_{0}^{T'}\int_{\R^N} \sigma   ( \chi_{\mathcal A(t)}- \chi_{A(t)}) \partial_t \vartheta \, \de x
		+ \int_{0}^{T'}\int_{\R^N } \sigma \vartheta  V  \, \de |\nabla \chi_{A(t)}|  \de t .
\end{align}
Moreover, we compute
\begin{align*}
	\int_{\R^N} \sigma   ( \chi_{\mathcal A(t)}- \chi_{A(t)}) (B \cdot \nabla )\vartheta \, \de x = 
&	 \int_{\R^N} \vartheta B \cdot n_{A(t)}  \, \de |\nabla \chi_{A(t)}|  - \int_{\R^N} (\nabla \sigma \cdot B) \vartheta   ( \chi_{\mathcal A(t)}- \chi_{A(t)}) \, \de x 		\\
&-	\int_{\R^N} \sigma \vartheta (\nabla \cdot B )   ( \chi_{\mathcal A(t)}- \chi_{A(t)})\, \de x . 
\end{align*}
Hence, by adding a zero, we obtain 
\begin{align*}
	& E_{\text{bulk}}[u_A|u_{\mathcal{A}} ](T') - E_{\text{bulk}}[u_A|u_{\mathcal{A}} ](0) \\
	&=  \int_{0}^{T'}\int_{\R^N} \sigma    ( \chi_{\mathcal A(t)}- \chi_{A(t)}) (\partial_t \vartheta + (B \cdot\nabla ) \vartheta) \, \de x +\int_{0}^{T'}\int_{\R^N } \sigma \vartheta ( V - B \cdot n_{A(t)} ) \, \de |\nabla \chi_{A(t)}|  \de t \\
&  \quad + \int_{0}^{T'} \int_{\R^N} (\nabla \sigma \cdot B+ \sigma \nabla \cdot B ) \vartheta    ( \chi_{\mathcal A(t)}- \chi_{A(t)}) \, \de x\de t  .
\end{align*}
The last integral can be directly bounded by the bulk energy functional \eqref{eqn:bulken} due to the regularity of $B$ and $\vartheta$.
Using the evolution equation \eqref{eqn:evvartheta} and the coercivity of $\vartheta$ given by \eqref{eqn:coervartheta}, the first term can {also} be bounded by the bulk energy functional \eqref{eqn:bulken}. 
{The second term is more complicated: Precisely, note that the left-hand side of \eqref{eq:relenineq} (which includes positive dissipation terms) is controlled from above by $C\int_0^{T'}E_{\rm rel}[u_A|u_{\mathcal{A}}]\, \de t$. Consequently, we use this information with Young’s inequality and the coercivity property \eqref{eqn:coervartheta} satisfied by the relative energy functional \eqref{eqn:relen} to find that 
$$ E_{\rm rel}[u_A|u_{\mathcal{A}}](T')+\int_{0}^{T'}\int_{\R^N } \sigma \vartheta ( V - B \cdot n_{A(t)} ) \, \de |\nabla \chi_{A(t)}|  \de t\leq E_{\rm rel}[u_A|u_{\mathcal{A}}](0) + C\int_0^{T'}E_{\rm rel}[u_A|u_{\mathcal{A}}]\, \de t.$$ 
Synthesizing the above inequalities we obtain \eqref{eqn:bulkenestimate}.}

{\textit{c) Weak-strong uniqueness.} }

{Adding \eqref{eqn:relenestimate} and \eqref{eqn:bulkenestimate}, we see that if $E_{\rm rel}[u_A|u_{\mathcal{A}}](0) + E_{\text{bulk}}[u_A|u_{\mathcal{A}} ](0) = 0$, then Gronwall's inequality implies $E_{\text{bulk}}[u_A|u_{\mathcal{A}} ](t) = 0$ for almost every $t\in [0,T)$. By construction of the bulk energy, this only happens if $A(t) = \mathcal{A}(t)$ for almost every $t\in [0,T)$, thereby recovering uniqueness of the flow.}
\end{proof}

\appendix
\section{Existence of Solutions to the Allen--Cahn equation}\label{sec:ACexist}

We discuss the existence of solutions of the heterogeneous Allen--Cahn equation \eqref{eqn:AChetero}, where we fix $\eps >0$ and assume for notational simplicity that the initial condition is $u_0\in H^1(\Omega)$. We only sketch the main ideas and refer to \cite[Lemmas 6-8]{HenselMoser} and \cite[Section 3]{LauxStinsonUllrich22} for complete proofs in related settings. With $E_\eps$ defined in \eqref{eqn:freeEnergy} and time-step $h>0,$ we introduce the minimizing movements scheme
\begin{equation}\label{eqn:minMovScheme}
u_i  \in \operatorname{argmin}_{u\in H^1(\Omega)} \Big\{\frac{1}{\eps}E_\eps [u]+\frac{1}{2h}\|u-u_{i-1}\|^2_{L^2(\Omega)} \Big\} \quad \text{ for } i = 1,\ldots \lceil T/h\rceil. 
\end{equation}
Note that if
\begin{equation}\label{eqn:monotonicity assumption}
\text{there is $C>0$ such that for all $x\in \Omega$, $W(x,u)$ is increasing for $u>C$ and decreasing for $u<-C$,}
\end{equation}  
then $\|u_i\|_{L^\infty(\Omega)}\leq C_0:=\max\{\|u_0\|, C\}$ for all $i>0$; to see this, compare $u_i$ with $\max\{\min\{u_i,C_0\},-C_0\}.$ Using Aubin--Lions--Simon type compactness results, it is direct to show that the piecewise constant and linear interpolations of $\{u_i\}_{i=0}^{\lceil T/h\rceil}$ converge to a function $u$ satisfying {items} (1) and (2) of Definition \ref{def:weakACsoln}. We remark that the fact $u \in L^2((0,T);H^2(\Omega))$ follows from elliptic regularity applied to the Neumann problem 
\begin{equation}\nonumber
\left\{\begin{aligned}
\Delta u & = \partial_t u + \frac{1}{\eps^2}\partial_u W(x,u)  && \text{ in }\Omega,\\
\nabla u\cdot n_{\Omega} & = 0  && \text{ on }\partial \Omega,
\end{aligned}\right.
\end{equation} at each time. Recovery of the energy dissipation equality of (3) in Definition \ref{def:weakACsoln} is more technical. Two standard options exist. Either one proves enough regularity for the solution so that the informal calculation 
$$\frac{\de}{\de t}E_\eps[u] = -\int_\Omega \eps|\partial_t u|^2\de x $$
becomes rigorous, or one supposes that $W$ can be perturbed to a convex function. We refer to proofs of optimal dissipation relations for the standard Allen--Cahn via regularity in \cite[Lemma 8]{HenselMoser} and for the anisotropic Allen--Cahn assuming that $W$ is $\lambda$-convex in \cite[Theorem 3.5]{LauxStinsonUllrich22}. Both of these strategies work in our setting since the sequence constructed above in \eqref{eqn:minMovScheme} is bounded in $L^\infty(\Omega)$ and $W\in C^2(\overline{\Omega}\times \R)$. Ultimately one recovers the desired existence result:

\begin{theorem}\label{thm:ACexist}
Let $N\geq 2$, $\Omega\subset \R^N$ be a $C^2$-domain, $W\in C^2(\overline{\Omega}\times \R)$ satisfy \eqref{eqn:monotonicity assumption}, and let $u_{0}\in H^1(\Omega)$. For $\eps>0$ fixed, there exists a weak solution $u$ of the heterogeneous Allen--Cahn equation \eqref{eqn:AChetero} with initial condition $u_0$ in the sense of Definition \ref{def:weakACsoln} such that in addition
\begin{equation}\nonumber
\|u_\eps\|_{L^\infty(\Omega\times (0,T))}\leq C_0 <\infty
\end{equation}
for a constant $C_0$ depending only on $\|u_0\|_{L^\infty(\Omega)}$ and $W$ (independent of $\eps>0$).
\end{theorem}

\section*{Acknowledgements}
L.G. was funded by the Deutsche Forschungsgemeinschaft–320021702/GRK2326– Energy, Entropy, and Dissipative Dynamics (EDDy). A.M. and K.S. were supported by funding from the Deutsche Forschungsgemeinschaft (DFG, German Research Foundation) under Germany’s Excellence Strategy – EXC-2047/1 – 390685813 and the DFG project 211504053 - SFB 1060. K.S. was also supported by funding from the NSF (USA) RTG grant DMS-2136198. 

\bibliographystyle{siam}
\bibliography{GT}

\end{document}